\documentclass{aims}
\usepackage{txfonts}
\def\typeofarticle{Journal article}
\def\currentvolume{x}
\def\currentissue{x}
\def\currentyear{2024}

\def\ppages{xxx--xxx}
\def\DOI{}
\def\Received{2023}
\def\Revised{2024}
\def\Accepted{April 2024}
\def\Published{2024}

\numberwithin{equation}{section}

\makeatletter
\renewcommand{\@biblabel}[1]{#1\hfill \hspace{-0.2cm}}
\makeatother

\usepackage{graphicx}
\usepackage{upref,ulem}
\usepackage{textcomp} 
\usepackage{tikz} 
\usepackage{pgfplots}
\usepackage{amsmath ,amssymb}
\usepackage{color,bbm,epsfig,float,fancyhdr}
\usepackage[mathscr]{eucal}
\usepackage{caption}
\newtheorem{example}{Example}[section]
\newtheorem{definition}{Definition}[section]
\newtheorem{theorem}{Theorem}[section]

\newtheorem{remark}{Remark}[section]
\newtheorem{proposition}{Proposition}[section]
\newtheorem*{maintheorem*}{Main Theorem}

\allowdisplaybreaks
\numberwithin{equation}{section}

\renewcommand{\i}{\ifmmode\mathit{\mathchar"7010 }\else\char"10 \fi}
\renewcommand{\j}{\ifmmode\mathit{\mathchar"7011 }\else\char"11 \fi}
\newcommand{\R}{\mathbb{R}}
\newcommand{\N}{\mathbb{N}}
\newcommand{\Z}{\mathbb{Z}}

\newcommand{\modulo}[1]{\left|#1\right|}
\newcommand{\norm}[1]{\left\|#1\right\|}

\newcommand{\BV}{\mathbf{BV}}
\renewcommand{\L}[1]{\mathbf{L^#1}}
\newcommand{\C}[1]{\mathbf{C^{#1}}}

{%

\begin{enumerate}}%
{\end{enumerate}}

{%

\begin{enumerate}}%
{\end{enumerate}}

\begin{document}




\title{A non-local traffic flow model for 1-to-1 junctions with buffer}

\author{%
	Felisia Angela Chiarello\affil{1},
	Jan Friedrich\affil{2}
	and
	Simone G\"ottlich\affil{3,}\corrauth
}

\shortauthors{the Author(s)}

\address{%
	\addr{\affilnum{1}}{DISIM, University of L'Aquila, Via Vetoio Ed. Coppito 1, 67100 L'Aquila, Italy.}
	\addr{\affilnum{2}}{IGPM, RWTH Aachen University, Templergraben 55, 52064 Aachen, Germany}
\addr{\affilnum{3}}{Department of Mathematics, University of Mannheim, B6, 28-29, 68159 Mannheim}}

\corraddr{goettlich@uni-mannheim.de.
}

\begin{abstract}
  \vspace{5pt}
In this paper, we introduce a non-local PDE-ODE traffic model devoted to the description of a 1-to-1 junction with buffer. We present an existence result in the free flow case as well as a numerical method to approximate weak solutions in the general case. In addition, we show a maximum principle which is uniform in the non-local interaction range.
Further, we exploit the limit models as the support of the kernel tends to zero and to infinity.
We compare them with other already existing models for traffic and production flow and present numerical examples.
\bigskip



\end{abstract}
 \maketitle                               
                                
\vspace{0.2cm}
\newlength{\fwidth}
\section{Introduction}
The aim of this paper is to study a non-local traffic flow model for 1-to-1 junctions. The term `non-local' refers to the dependence of the flux function on a convolution term between a kernel function and the velocity function depending on the conserved quantity.
This type of conservation laws are able to describe several phenomena arising in many fields of application and for this reason received growing attention in the recent years.
The first macroscopic traffic flow model based on fluid-dynamics equations has been introduced in the transport literature with the Lighthill, Whitham and Richards (LWR) model \cite{lwr_1, lwr_2}, that consists in one scalar equation expressing  conservation of cars.
Since then, several approaches have been developed during the last years, addressing the need for more sophisticated models to better understand the nonlinear traffic dynamics. 
 `Non-local' versions of the LWR model have been recently proposed in \cite{BlandinGoatin2016, chiarello2018global, friedrich2018godunov}. In this type of models, the flux is assumed to depend on a weighted mean of the downstream traffic density or velocity. 
These kind of non-local traffic models are intended to describe the behaviour of drivers that adapt their velocity with respect to what happens in front of them. For this reason, the flux function  depends on a `downstream' convolution term between the density of vehicles or the velocity, and a kernel function supported on the negative axis, such that drivers only look forward, not backward. 
There are general existence and uniqueness results for non-local equations in \cite{AmorimColomboTexeira, KeimerPflug2017, chiarello2018global, friedrich2018godunov} for scalar equations in one-space dimension, in \cite{ColomboHertyMercier2011, KEIMER2018} for multi-dimensional scalar equations, and in \cite{ACG2015} for the multi-dimensional system case.
There are mainly two different approaches to prove the existence of solutions for these non-local models. One is providing suitable compactness estimates on a sequence of approximate solutions constructed through finite volume schemes, as in \cite{AmorimColomboTexeira, ACG2015}. Another approach relies on the method of characteristics and fixed-point theorems, as proposed in \cite{KeimerPflug2017, KEIMER2018}. 
In \cite{AmorimColomboTexeira}, Kru\v{z}kov-type entropy conditions are used to prove the $\L1-$stability with respect to the initial data through the doubling of variable technique, while in \cite{KeimerPflug2017,KEIMER2018}, the uniqueness of weak solutions is obtained directly from the fixed-point theorem, without prescribing any kind of entropy condition.

One key question which naturally arises in the non-local setting from a mathematical, but also modelling point of view, is the model behavior for the two limits of the non-local interaction range. 
Considering an infinite range, the non-local scalar traffic flow model tends to a linear transport equation \cite{chiarello2018global}.
Contrary, if the range tends to zero, 
non-local scalar traffic flow models converge to the local LWR model under specific assumptions, e.g.~\cite{keimer2019approximation, friedrich2022conservation,coclite2020general,colombo2022nonlocal, bressan2020entropy, bressan2019traffic} or~\cite{ChiarelloKeimer2023} for the singular non-local-to-local limit for weakly-coupled systems.

Some works about non-local traffic models deal with networks, for example~\cite{CamilliDT18, friedrich2020onetoone, ChiarelloCocliteDisc, friedrich2022network, shen2019stationary, shen2019traveling, keimer2018bounded}. 
In~\cite{CamilliDT18}, the authors consider measure valued solutions for non-local transport equations and~\cite{keimer2018bounded} deals with non-local conservation laws on bounded domains while 
\cite{friedrich2020onetoone, ChiarelloCocliteDisc, shen2019traveling, shen2019stationary} include 1-to-1 junctions. In particular in~\cite{friedrich2020onetoone, ChiarelloCocliteDisc}, the existence and well-posedness of solutions at a 1-to-1 junction is shown, where the
roads are allowed to differ in the speed limits and maximum road capacities.
The limit of those models for a non-local range tending to infinity is investigated in \cite{friedrich2022network}.
Here, the models still converge to a linear transport equation which includes a capacity constraint on the flux induced by the coupling of the roads.
In contrast, the limit to zero is only investigated numerically so far and the correct limit model is still an open problem.

In this paper we aim to model a 1-to-1 junction with a buffer, generalizing the `local' model in~\cite{herty2009novel} to the non-local setting. Further, we want to exploit the limit models for the non-local range tending to infinity and zero.
The paper is organized as follows. In Section 2, we introduce a new non-local model for 1-to-1 junctions with buffer. Then, in Section 3 we investigate analytically the existence of weak solutions in the free flow case, i.e. when the buffer is not active on the first road, by introducing the characteristics of the modeling equations.
In Section 4 we perform a numerical discretisation of our model using an upwind-type numerical scheme. In Section 5, the limit model as the length of the support $\eta$ of the kernel function goes to zero is exploited and in Section 6 the opposite limit model as the length of the support $\eta$ of the kernel function goes to infinity is investigated. In Section 7, we give a conclusion and mention open problems. 

\section{Modeling equations}
As aforementioned, we consider a non-local traffic flow model for 1-to-1 junctions that includes a buffer.
Therefore, let us recall first the 1-to-1 junction model from \cite{friedrich2020onetoone}.
This model does not consider a buffer so far and the dynamics for the traffic density on each road are given by
\begin{equation}
\begin{aligned} \label{eq:1to1NObuffer} 
&\partial_t \rho_1(t,x) + \partial_x \left(\rho_1(t,x) V_1(t,x)+\min\{\rho_1 V_2(t,x),\, \rho^2_{\max} V_2(t,x)\}\right)=0,\\
&\partial_t \rho_2(t,x) + \partial_x (\rho_2(t,x) V_2(t,x))=0.
\end{aligned}
\end{equation}
Here, $V_1$ and $V_2$ are the parts of the non-local velocity on each road segment defined by 
\begin{equation*}
    V_1(t,x):=\int_{\min\{x,\,0\}}^{\min\{x+\eta,\,0\}}v_1(\rho_1(t,y))\omega_\eta(y-x)\,dy,
\end{equation*}
\begin{equation*}
    V_2(t,x):=\int_{\max\{x,\,0\}}^{\max\{x+\eta,\,0\}}v_2(\rho_2(t,y))\omega_\eta(y-x)\,dy.
\end{equation*}
In particular, the junction point is placed at $x=0$, $\rho_1$ and $\rho_2$ are the traffic densities, $\eta>0$ is the non-local range and support of the kernel function $\omega_\eta$, which will be detailed below, and $v_1$ and $v_2$ are suitable velocity functions.
The most important feature is the coupling between the two road segments which is simply induced by the minimum operator and the maximum capacity of the second road.
The latter one expresses the supply of the second road.
At the junction itself, the second road can receive a flux up to $\rho^2_{\max} V_2(t,0)$.
Due to non-locality, this capacity restriction must be considered already in the interval $[-\eta,0)$.
From a microscopic point of view, this means that as soon as a driver is aware of an intersection, she/he takes the associated (possible) restrictions into account.

We now extend this model by placing a buffer between the two roads.
These buffers could, for example, model a simplified version (by neglecting the exact geometry) of highway on-ramps or roundabouts.
A buffer has a maximum capacity $\mu$, i.e. the maximum rate at which cars can enter or leave the buffer.
Furthermore, the buffer is full at $r_{\max}$.
This creates a specific supply that the buffer can provide for the flow of the first road.
Similar to the situation without a buffer, the supply must be taken into account as soon as a driver sees the buffer.
As in the situation without buffer, the second road at the junction can take a maximum flow of $\rho^2_{\max} V_2(t,0)$.
This results in the following modeling equations for the buffer $r(t)$ and the traffic densities:
\begin{equation}
\begin{aligned} \label{eq:model} 
&\partial_t \rho_1(t,x) + \partial_x \left(\rho_1(t,x) V_1(t,x)+\min\{\rho_1 V_2(t,x),\, s_B(t,x)\}\right)=0,\\
&\partial_t \rho_2(t,x) + \partial_x (\rho_2(t,x) V_2(t,x))=0,\\
& r'(t)=\min\{s_B(t,0^-),\,\rho_1(t, 0^-) V_2(t,0)\}-\min\{d_B(t),\,\rho^2_{\max} V_2(t,0)\},
\end{aligned}
\end{equation}
where 
\begin{align*}
&s_B(t,x)=\begin{cases}
    \mu \int_0^{\max\{x+\eta,0\}}\omega_\eta(y-x)\, dy, \quad &r(t)\in [0, r_{\max}),\\
    \min\{\rho^2_{\max} V_2(t,x),\, \mu \int_0^{\max\{x+\eta,0\}}\omega_\eta(y-x)\, dy\}, &r(t)= r_{\max},
\end{cases}\\
&d_B(t)=\begin{cases}
        \mu, \quad &r(t)\in (0, r_{\max}],\\
    \min\{\rho_1(t,0^-) V_2(t,0),\,\mu\}, &r(t)= 0.
\end{cases}
\end{align*}
Note that the model \eqref{eq:model} is inspired by its local counterpart introduced in \cite{herty2009novel}.
For a better readability we postpone the detailed discussion of the local variant to Section \ref{subsec:limit0model}.

To close the equations~\eqref{eq:model} we couple them with the initial data
\begin{equation}
\begin{aligned} \label{eq:hyp} 
\rho_1(0,x)=\rho_{1,0}(x) \in \BV((-\infty,0)),\quad
\rho_2(0,x)=\rho_{2,0}(x) \in \BV((0,+\infty)),\quad
r(0)=r_0\geq 0,
\end{aligned}
\end{equation} 
where $\BV([a,b]):=\{f\in \L1([a,b]) |\: TV(f)<+\infty \}.$
For the well-posedness of the model, we assume the following common assumptions, see e.g. \cite{BlandinGoatin2016, chiarello2018global, friedrich2020onetoone, friedrich2022network, friedrich2018godunov}
\begin{itemize}
\item[\bf (H1)] 
	$\quad \omega_{\eta}\in \C1([0,\eta];\R^+),\,\eta>0, \, \omega'_\eta\leq 0,$ 
 $\int_0^{\eta}\omega_\eta(x)dx=1$; 
\item[\bf (H2)] 
	$\quad v_e:\R^+\rightarrow \R^+$, is a smooth non-increasing function such that $v_e(0)=V_e^{\max}>0$ and
	$v_e(\rho)=0$ for $\rho\geq \rho_{\max}^e$ for $e\in\{1,2\}$. 
\end{itemize}
Inspired by~\cite[Definition 2.1]{DelleMonacheGoatin2014} and~\cite[Definition 4.1]{DELLEMONACHE2014}, solutions are intended in the following sense: 
\begin{definition}\label{def:weaksol}
A couple of functions $\rho_e\in C(\R^+, \L1_{loc}((a_e,b_e)))$ with $e\in\{1,2\},$ and $a_1=-\infty, b_1=0$ and $a_2=0, b_2=+\infty,$ is called a weak solution of~\eqref{eq:model}-\eqref{eq:hyp} if 
\begin{enumerate}
    \item for $e=1,\, \rho_e$ is a weak solution on $(-\infty, 0)$  to 
    $$\partial_t \rho_e(t,x) + \partial_x \left(\rho_e(t,x) V_1(t,x)+\min\{\rho_e V_2(t,x),\, s_B(t,x)\}\right)=0; $$
    \item for $e=2,\, \rho_e$ is a weak solution on $(0, +\infty)$  to 
    $$\partial_t \rho_e(t,x) + \partial_x (\rho_e(t,x) V_2(t,x))=0; $$
    \item for all $t>0,$ it holds 
    $$\rho_2(t,0^+) V_2(t,0)=\min\{d_B(t),\,\rho^2_{\max} V_2(t,0)\}; $$
    \item for $e\in\{1,2\}$ $\rho_e\in \L\infty((0,T)\times (a_e,b_e)$, 
    \item $r(t)$ is a Carath\`eodory solution, i.e. for a.e. $t\in\R^+$ $$r(t)=r_0+\int_0^t  \left(\min\{s_B(s,0^-),\,\rho_1(s, 0^-) V_2(s,0)\}-\min\{d_B(s),\,\rho^2_{\max} V_2(s,0)\}\right) \,ds.$$
   \end{enumerate}
\end{definition}

\section{Existence of solutions for free flow}
In this section we exploit the existence of weak solutions in a particular case, i.e. the flow on the first road is not influenced by the buffer and can pass onto the second road if the buffer is empty.
In case of a non-empty buffer, its size is decreasing over time. Even though this is a rather simple example, it represents the case that in the past the buffer filled due to congestion, but now the traffic jams on the roads resolved to allow free flow such that the buffer gets empty.\\
We introduce the characteristic lines of our model, see \cite[Definition 2.5 and Remark 2.6]{keimer2018bounded}.
For a given non-local term $W_1\in \L\infty((0,T);W^{1,\infty}(-\infty,0))$ the characteristics $\xi_{W_1}^1$ for $(t,x)\in(0,T)\times (-\infty,0)$ are given by
\begin{equation}\label{eq:char_1}
\xi_{W_1}^1[t,x](\tau)=x+\int_t^\tau W_1(s,\xi^1_{W_1}[t,x](s))ds,\quad \tau \in[0,T].
\end{equation}
Similar for a given non-local term $W_2\in \L\infty((0,T);W^{1,\infty}(0,\infty))$ the characteristics $\xi_{W_2}^2$ for $(t,x)\in(0,T)\times (0,\infty)$ on the second road are given by
\begin{equation}\label{eq:char_2}
\xi_{W_2}^2[t,x](\tau)=x+\int_t^\tau W_2(s,\xi^2_{W_2}[t,x](s))ds,\quad \tau \in[0,T].
\end{equation}
For $(t_0,x_0)\in\{(t,x)\in(0,T)\times(0,\infty): x\leq \xi_{W_2}^2[0,0](t)\}$ and $x_1\in[0,x_0]$ the time inverted characteristics on the second road are defined 
\[ \xi_{W_2}^2[t_0,x_0]^{-1}_{\max}(x_1)=\max\{\tau\in[0,t_0]:\xi_{W_2}^2[t_0,x_0](\tau)=x_1\},\]
and the non-local terms are
\begin{align*}
W_1[\rho_1,\rho_2](t,x)=\begin{cases}
V_1(t,x)+V_2(t,x),\ &\text{ if }\rho_1(t,x)V_2(t,x)\leq s_B(t,x),\\
V_1(t,x)&\text{ else,}
\end{cases} \quad W_2[\rho_1,\rho_2](t,x)=V_2(t,x).
\end{align*}
Note that $V_1$ and $V_2$ depend on $\rho_1$ and $\rho_2$, respectively.
Moreover, the non-local term on the first road differs from the non-local quantity in~\eqref{eq:model}.
In case the supply $s_B$ is present on the first road, this term is independent of the density and does not influence the characteristic equation.
Nevertheless, it plays a source term role when the solution is written in terms of the characteristics.
This is similar to~\cite[equation (3.1)]{keimer2018bounded}.\\
We can prove the following existence result in the free flow case.
\begin{theorem}\label{thm:ex}
Let $v_e(\rho)=v_{\max}^e(1-\rho/\rho_{\max}^e)$ for $e\in\{1,2\}$, $v_{\max}^2\leq v_{\max}^1$ and $\mu\geq \rho_{\max}^1 v_{\max}^2$.
\begin{enumerate}
\item Let $\rho_{\max}^2\geq \rho_{\max}^1$, then a weak solution in the sense of Definition \ref{def:weaksol} exists.
\item Let $\rho_{\max}^2<\rho_{\max}^1$ and $r_0\in (0, r_{\max})$, then the existence of a weak solution in the sense of Definition \ref{def:weaksol} is guaranteed as long as $r(t)\in (0,r_{\max})$, e.g. for $t\in [0,T)$ with $T=\min\{r_0/(\rho_{\max}^2v_{\max}^2),(r_{\max}-r_0)/(\rho_{\max}^1v_{\max}^2)\}$.
\end{enumerate}
Furthermore, the weak solutions satisfy the following maximum principle:
\begin{align*}
&0\leq \rho_{1}(t,x)\leq \rho^1_{\max}\ \text{ for a.e. }(t,x)\in \R^+\times (-\infty,0)\\
&0\leq \rho_{2}(t,x)\leq \rho^2_{\max}\ \text{ for a.e. }(t,x)\in \R^+\times (0,\infty).
\end{align*}
\end{theorem}
\begin{proof}
The proof is based on existence results from \cite{keimer2018bounded} and \cite{friedrich2020onetoone}.\\
We start with the first case, i.e. empty buffer $r_0=0$. The assumption $\mu\geq \rho_{\max}^1 v_{\max}^2$ guarantees that the buffer stays empty, since $r'(t)=0$ holds for $t\geq 0$.
Furthermore, due to the same assumption the supply function is not present in the flux of the first road.
Hence, the model is given by \eqref{eq:1to1NObuffer} and the existence of a weak solution satisfying the maximum principle is given by \cite[Theorem 1]{friedrich2020onetoone}.\\
Now, let $r_0>0$.
Here, the buffer size is decreasing until it is empty, since for $r(t)>0$
\[r'(t)=\rho_1(t,0-)V_2(t,0)-\min\{\mu,\rho_{\max}^2V_2(t,0)\}<0\]
holds.
In particular, as long as the buffer is positive, the inflow on the second road is given by $\min\{\mu,\rho_{\max}^2V_2(t,0)\}$, independently of $\rho_1$.
Let $T^*$ be the time at which the buffer becomes empty.
One can view the whole problem as two separate problems on semi-infinite intervals.
One initial boundary value problem reads as
\begin{align}\label{eq:surrogatemodel}
\partial_t \rho_2(t,x)+\partial_x (\rho_2(t,x) V_2(t,x))&=0,\quad &(t,x)\in (0,T^*)\times(0,\infty),\\
\rho_2(0,x)&=\rho_{2,0}(x),&x\in(0,\infty),\nonumber \\
V_2(t,0)\rho_2(t,0)&=\min\{\mu,\rho_{\max}^2V_2(t,0)\},&t\in(0,T^*).\nonumber
\end{align}
Due to the linear velocity this model is a specific choice of the modeling framework in \cite{keimer2018bounded}.
Note that due to the choice of the inflow the boundary datum remains well-posed in the case $V_2(t,0)=0$, see \cite[Remark 2.2]{keimer2018bounded}.\\
Hence, existence of weak solutions and the maximum principle are guaranteed by \cite[Corollary 5.4, 5.9 and Remark 5.6]{keimer2018bounded} for the problem \eqref{eq:surrogatemodel}.
We can proceed similarly for the first road by choosing the boundary data appropriately.
Due to the non-locality we need to choose a function which depends on $\rho_2$ and gives the same dynamics as model \eqref{eq:model}.
We can choose the following surrogate model for $(t,x)\in (0,T^*)\times(-\infty,0)$
\begin{align*}
\partial_t \rho_1(t,x)+\partial_x \left(\rho_1(t,x) v_1\left(\int_x^{x+\eta} \omega_\eta(y-x) \left(\begin{cases} \rho_1(t,y)\quad & y\in(-\infty,0)\\ v_1^{-1}(v_2(\rho_2(t,y)) &\text{else}\end{cases}\right)dy\right)\right)&=0.
\end{align*}
Here, we need the assumption $v_{\max}^2\leq v_{\max}^1$ such that the inverse is well defined and the boundary data is non-negative.
Again this fits into the framework of \cite{keimer2018bounded} such that the existence of solutions is guaranteed for $t\in(0,T^*)$.\\
For $t\geq T^*$ the buffer is empty and we are back in the previous case ($r_0=0$) such that existence is given by~\cite{friedrich2020onetoone}.\\
The proof of the second case is similar to the case $r_0>0$ and $t\in(0,T^*)$ discussed above with the slight modification that for $r(t)=r_{\max}$ or $r(t)=0$ the dynamics changes.
Hence, the given time $T=\min\{r_0/(\rho_{\max}^2v_{\max}^2),(r_{\max}-r_0)/(\rho_{\max}^1v_{\max}^2)\}$ is a simple lower bound on $T^*$ that guarantees $r(t)\in(0,r_{\max})$ for $t\in[0,T)$.
\end{proof}
Let us shortly discuss how the results can be extended.
\begin{remark}[Uniqueness]
In Theorem \ref{thm:ex} we only discussed the existence of solutions for the free flow case. 
As long as $r(t)\in(0,r_{\max})$ we also obtain the uniqueness, thanks to \cite{keimer2018bounded}.
In \cite{friedrich2020onetoone} only the uniqueness of a weak entropy solution has been shown.
Nevertheless, an extension of these results to the uniqueness of weak solutions with similar techniques as in \cite{keimer2018bounded,friedrich2022conservation,friedrich2023conservation2} can be expected, which would provide uniqueness of weak solutions for the cases considered in Theorem~\ref{thm:ex}.
\end{remark}
\begin{remark}[Nonlinear velocities]
To obtain the above results for nonlinear velocity functions the uniqueness of weak solutions for the so-called non-local in velocity model~\cite{friedrich2018godunov} on a bounded domain needs to be shown.
This can be achieved by following the lines of \cite{keimer2018bounded} with similar modifications as done in \cite{friedrich2022conservation,friedrich2023conservation2}.
\end{remark}
\begin{remark}[General model]\label{rem:completemodel}
The existence and uniqueness for the complete model~\eqref{eq:model} can be proven by exploiting the characteristics as in~\cite{keimer2018bounded}. In particular, the solution on each interval can be written in terms of the characteristics, which can be plugged into the non-local terms. This gives a fixed-point problem with respect to the two non-local terms.
The difficulty here and the main difference with~\cite{keimer2018bounded} is that we have to deal with two non-local terms which are coupled.
The non-local term on the first road depends on the solution of the second road and the solution on the second road depends on the first one through the outflow of the buffer.
This makes the analysis of~\cite{keimer2018bounded} more complicated.
\end{remark}

\section{Numerical discretisation}
In the previous section we have seen that in specific cases we can guarantee the existence of solutions and a maximum principle. The latter will be now extended to the general case.
Therefore, we will approximate the solutions by using a numerical scheme.
This scheme is based on the numerical schemes introduced in \cite{friedrich2020onetoone,friedrich2018godunov,friedrich2022network,friedrich2023numerical}.

For $j\in \Z$ and $n \in \N$ and $e\in E:=\{1,2\}$, let $x_{e,j-1/2} = j\Delta x$ be the cell interfaces, $x_{e,j} = (j+1/2)\Delta x$ the
cells centers, corresponding to a space step $\Delta x$ such that $\eta = N_\eta \Delta x$ for some $N_\eta \in \N$,
and let $t_n = n\Delta t$ be the time mesh. In particular, $x = x_{e,-1/2} = 0$ is a cell interface. 
Note that we have $j\geq 0$ if $e=2$ and $j<0$ if  $e=1.$
We aim at constructing a finite volume approximate solution $\rho^{\Delta x}_e$ such that $\rho^{\Delta x}_e (t,x) = \rho^n_{e, j}$
for
$(t,x) \in [t^n, t^{n+1}) \times [x_{e,j-1/2}, x_{e,j+1/2})$. To this end, we approximate the initial datum $\rho_0$ with the
cell averages
\begin{equation*}
    \rho_{e, j}^0=\frac{1}{\Delta x} \int_{x_{e,j-1/2}}^{x_{e,j+1/2}} \rho_{e, 0}(x) \,dx, \qquad \begin{cases}
    j< 0, \quad e=1,\\
    j\geq 0, \quad e=2.
    \end{cases}
\end{equation*}

We consider the following numerical fluxes similar to \cite{friedrich2020onetoone,friedrich2022network}
\begin{equation}\label{eq:flux}
    F^n_{e,j}:=\begin{cases}
  \rho^n_{e,j} V^n_{e,j}+\min\left\{ \rho^n_{e,j} V^n_{2,j}, s^n_{B,j}  \right\}, \quad &e=1,\\
   \rho^n_{e,j} V^n_{e,j}, \quad &e=2,
    \end{cases}
\end{equation}
with 
\begin{align*}
  V^{n}_{1,j}&=\sum_{k=0}^{\min\{-j-2,\,N_\eta-1\}} \gamma_k v_1(\rho^n_{1,\,j+k+1}), \quad    V^{n}_{2,j}=\sum_{k=\max\{-j-1,\,0\}}^{N_{\eta-1}} \gamma_k v_2(\rho^n_{2,\,j+k+1}),\\
 s^n_{B,j}&= \begin{cases}
    \mu \displaystyle{\sum_{k=-j-1}^{N_{\eta-1}}}\gamma_{k}, \quad &r^n\in [0, r_{\max}),\\
    \min\left\{\rho^2_{\max} V^n_{2, j},\, \mu \displaystyle{\sum_{k=-j-1}^{N_{\eta-1}}}\gamma_{k}\right\}, &r^n= r_{\max}, 
\end{cases}\quad
d^n_B=\begin{cases}
        \mu, \quad &r^n\in (0, r_{\max}],\\
    \min\{\rho^n_{1,-1} V^n_{2,-1},\,\mu\}, &r^n= 0.
\end{cases}
\end{align*}
where $\gamma_k=\int_{k \Delta x}^{(k+1)\Delta x} \omega_\eta(x)\,dx, \ k=0,...,N_\eta-1$ and $r^n$ is the approximation of the buffer computed by an explicit Euler scheme.
In this way, we define the following finite volume scheme: 
\begin{equation}\label{eq:scheme}
\begin{aligned}
    \rho^{n+1}_{e,j}&= \rho^{n}_{e,j}-\frac{\Delta t}{\Delta x} \left(F^n_{e,j}-F^n_{e,j-1}\right),\\
    r^{n+1}&=r^{n}+\Delta t \left(\min\{s^n_{B,0}, \rho^n_{1,0}V^n_{2,0}\}-\min\{d^n_B, \rho^2_{\max}V^n_{2,0}\}\right),
    \end{aligned}
    \quad n\in \N.
\end{equation}

We now show that the numerical scheme satisfies a reasonable maximum principle on each road.
\begin{proposition}\label{prop:maximum_principle}
Under hypothesis~\eqref{eq:hyp} and the CFL condition \begin{equation} \label{eq:CFL}
    \frac{\Delta t}{\Delta x}\leq \frac{1}{\gamma_0 \norm{v'} \norm{\rho}+2\norm{v}},
\end{equation} 
with $\norm{v}=\max\{\norm{v_1},\norm{v_2}\},\, \norm{v'}=\max\{\norm{v_1'},\norm{v_2'}\}$ and $\norm{\rho}=\max\{\rho^1_{\max},\,\rho^2_{\max}\},$
the sequence generated by the numerical scheme~\eqref{eq:flux}-\eqref{eq:scheme} satisfies the following maximum principle:
$$0\leq \rho^{n}_{e,j}\leq \rho^e_{\max}.$$
\end{proposition} 
\begin{proof}
Let us start by showing the positivity, we directly obtain 
$$ \rho^{n+1}_{e,j}= \rho^{n}_{e,j}-\frac{\Delta t}{\Delta x} \left(F^n_{e,j}-F^n_{e,j-1}\right)\geq \rho^{n}_{e,j}-\frac{\Delta t}{\Delta x} \norm{v} \rho^n_{e,j}\geq 0,$$
with $\norm{v}=\max\{\norm{v_1},\norm{v_2}\},$ using $\min\left\{ \rho^n_{e,j} V^n_{2,j}, s^n_{B,j}  \right\}\leq \rho^n_{e,j} V^n_{2,j}.$
For the upper bound, it is easy to see that 
\begin{align*}
    V^n_{1,j-1}- V^n_{1,j} \leq \begin{cases}
    \gamma_0 v_1(\rho^n_{1,j}), \quad &j\leq -1,\\
    0, \quad &j\geq 0,
    \end{cases}\quad
    V^n_{2,j-1}- V^n_{2,j} \leq \begin{cases}
    0, \quad &j\leq -1,\\
    \gamma_0 v_2(\rho^n_{2,j}), \quad &j\geq 0.
    \end{cases}
\end{align*}
Using $v_1(\rho^1_{\max})=v_2(\rho^2_{\max})=0$ and the mean value theorem, we get 
\begin{align*}
    V^n_{1,j-1}- V^n_{1,j} \leq \begin{cases}
    \gamma_0 \norm{v_1'}(-\rho^n_{1,j}+\rho^1_{\max}), \quad &j\leq -1,\\
    0, \quad &j\geq 0,
    \end{cases}\\
    V^n_{2,j-1}- V^n_{2,j} \leq \begin{cases}
    0, \quad &j\leq -1,\\
    \gamma_0  \norm{v_2'}(-\rho^n_{2,j}+\rho^2_{\max}), \quad &j\geq 0.
    \end{cases}
\end{align*}
Considering the case $j\leq -1,$ multiplying the first inequality by $\rho^1_{\max}$ and subtracting $V^n_{1,j}\rho^n_{1,j},$ we get  
\begin{equation}\label{eq:min_1}
    V^n_{1,j-1} \rho^1_{\max}- V^n_{1,j}\rho^n_{1,j}\leq  \left(\gamma_0 \norm{v_1'}\norm{\rho_1}+V^n_{1,j}\right) (-\rho^n_{1,j}+\rho^1_{\max}). 
\end{equation}
Similarly, writing 
\begin{align*}
    s^n_{B,j-1}- s^n_{B,j} = \begin{cases}
     -\mu \gamma_{-j-1} 
    \quad &r^n\in [0,r_{\max}),\\
    \min\{\rho^2_{\max}V^n_{2,j-1}, \mu \sum_{k=-j}^{N_\eta-1} \gamma_k\}-\min\{\rho^2_{\max}V^n_{2,j}, \mu \sum_{k=-j-1}^{N_\eta-1} \gamma_k\} \quad &r^n=r_{\max},
    \end{cases} \leq 0,
\end{align*}
because \begin{align*}
    &\min\left\{\rho^2_{\max}V^n_{2,j-1}, \mu \sum_{k=-j}^{N_\eta-1} \gamma_k\right\}-\min\left\{\rho^2_{\max}V^n_{2,j}, \mu \sum_{k=-j-1}^{N_\eta-1} \gamma_k\right\} \\
    &\leq\min\left\{\rho^2_{\max}V^n_{2,j}, \mu \sum_{k=-j}^{N_\eta-1} \gamma_k\right\}-\min\left\{\rho^2_{\max}V^n_{2,j}, \mu \sum_{k=-j-1}^{N_\eta-1} \gamma_k\right\}\\
    &=\frac{\mu \sum_{k=-j}^{N_\eta-1} \gamma_k-\mu \sum_{k=-j-1}^{N_\eta-1} \gamma_k-\modulo{\rho^2_{\max}V^n_{2,j}-\mu \sum_{k=-j}^{N_\eta-1} \gamma_k}+\modulo{\rho^2_{\max}V^n_{2,j}-\mu \sum_{k=-j-1}^{N_\eta-1} \gamma_k}}{2}\\
    &\leq \frac{\mu \sum_{k=-j}^{N_\eta-1} \gamma_k-\mu \sum_{k=-j-1}^{N_\eta-1} \gamma_k+\modulo{\mu \sum_{k=-j}^{N_\eta-1} \gamma_k-\mu \sum_{k=-j-1}^{N_\eta-1} \gamma_k}}{2}=0.
    \end{align*}
Then, we get  
\begin{equation}\label{eq:min_2}
\begin{aligned}
    \min\{\rho^1_{\max} V^n_{2,j-1},s^n_{B,j-1}\}&-\min\{\rho^n_{1,j} V^n_{2,j},s^n_{B,j}\}\\
    &\leq 
   \min\{\rho^1_{\max} V^n_{2,j},s^n_{B,j}\}-\min\{\rho^n_{1,j} V^n_{2,j},s^n_{B,j}\}\\
   &\leq\left(\rho^1_{\max}-\rho^n_{1,j}\right)V^n_{2,j}.
   \end{aligned}
 \end{equation}

Adding the two inequalities~\eqref{eq:min_1}-\eqref{eq:min_2}, due to the CFL condition~\eqref{eq:CFL}, we have 
$\rho^{n+1}_{1, j}\leq \rho^1_{\max}.$ \\
For $j\geq 0,$ we can write the following bound
$$ \rho^2_{\max} V^{n}_{2,j-1} - \rho^n_{2,j} V^{n}_{2,j} \leq (\gamma_0 \norm{v'} \norm{\rho}+\norm{v}) (\rho^2_{\max}-\rho^n_j),$$ that follows analogously to what has been done above. 
In particular, the bound holds for $j=0$, since for the inflow on road 2 we have
$$ \min\{d_B^n,\rho^2_{\max} V^{n}_{2,-1}\}\leq \rho^2_{\max} V^{n}_{2,-1}.$$
Then, we get 
$$\rho^{n+1}_{2, j}\leq \rho^{n}_{2, j}+\frac{\Delta t}{\Delta x}\left(\rho^2_{\max} V^n_{2,j-1}-\rho^n_{2,j} V^n_{2,j}\right)\leq\rho^2_{\max}. $$
This concludes the proof.
\end{proof}
\begin{remark}
We point out that the maximum principle in Proposition \ref{prop:maximum_principle} is uniform with respect to $\eta$ and this allows us to say that if they exist, the limit solutions as $\eta \to 0$ and $\eta \to +\infty$ satisfy the maximum principle, too.
\end{remark}
In the following we are in particularly interested in the model hierarchies induced by the non-local range $\eta$.
Therefore, we will investigate the models for $\eta\to 0$ and $\eta \to \infty$ in the next sections.

\section{Limit to zero}\label{subsec:limit0model}
The singular local limit is a very interesting topic regarding non-local conservation laws.
The problem is defined as follows. In the scalar case, a parameter $\eta>0$ related to the support of the kernel is fixed and the re-scaled kernel function is considered
\begin{equation*}
    \omega_\eta(x)=\frac{1}{\eta}\omega\left(\frac{x}{\eta}\right).
\end{equation*}
Owing to the assumption $\int_{\R} \omega_\eta(x)dx=1,$ when $\eta\to0^+$ the family $\omega_\eta$ converges weakly$^*$ in the sense of measures to the Dirac delta and we formally obtain the corresponding local conservation law:
\begin{equation*}
 \textbf{Non-local: }  \begin{cases} \partial_t \rho_\eta+ \partial_x(\rho_\eta (v(\rho_\eta)\ast\omega_\eta))=0,\\ \rho_\eta(0,x)=\bar \rho(x), \end{cases} \to  \textbf{ Local: }\begin{cases} \partial_t \rho+ \partial_x(\rho v(\rho))=0,\\ \rho(0,x)=\bar \rho(x),\end{cases}
\end{equation*}
with $v:\R \to \R$ being a Lipschitz continuous vector-valued function.
The above derivation is just formal and has to be rigorously justified. The question is whether the solution $\rho_\eta$ of the non-local Cauchy problem converges to the entropy admissible solution of the local Cauchy problem in some suitable topology. The rigorous derivation of this limit with the convolution inside $v$, i.e. $v(\rho_\eta \ast \omega_\eta),$ was
initially posed in \cite{AmorimColomboTexeira} motivated by numerical evidence, later corroborated in \cite{BlandinGoatin2016}. The answer is positive if we consider an even kernel function and a smooth compactly supported initial datum, as proved in  \cite{Zumbrun1999}. The solution of one-dimensional scalar non-local conservation laws with an anisotropic kernel converges to the entropy solution of the corresponding local conservation laws also in the case of monotone initial data, see~\cite{keimer2019approximation, friedrich2022conservation}, for an exponential-type kernel~\cite{coclite2020general, friedrich2022conservation,bressan2020entropy} or 
considering a positive initial datum and a convex kernel~\cite{colombo2022nonlocal}.
Nevertheless, in general, the solution of the non-local Cauchy problem does not converge to the solution of the local one and in \cite{ColomboCrippaSpinoloontheSingular} the authors show three counterexamples in this sense, while in \cite{ColomboCrippaGraffSpinolo}, the role of numerical viscosity in the numerical investigation of such non-local-to-local limit is highlighted.\\

\subsection{A local buffer model}
A local buffer model which strongly inspired the non-local model \eqref{eq:model} is originally introduced in \cite{herty2009novel}.
For completeness we recall the model of \cite{herty2009novel}.
Note that this model is derived only for the same flux function on each road, but an extension to different flux functions is straightforward.
The modeling equations are given by
\begin{equation}
\begin{aligned} \label{eq:model0herty} 
&\partial_t \rho_1(t,x) + \partial_x \left(\rho_1(t,x) v_1(\rho_1(t,x))\right)=0,\\
&\partial_t \rho_2(t,x) + \partial_x (\rho_2(t,x) v_2(\rho_2(t,x)))=0.
\end{aligned}
\end{equation}
Since we consider a local model we only need to describe the fluxes at the intersection, which couples the two equations.
Following \cite{herty2009novel} this coupling is given by
\begin{align*}
    &q_1(t)=\min\{s_B(t),D_1(\rho_1(t,0^-))\},\qquad q_2(t)=\min\{d_B(t),S_2(\rho_2(t,0^+)\}\\
    &s_B(t)=\begin{cases}
        \mu,&\text{if }r(t)\in[0,r_{\max}),\\
        \min\{S_2(\rho_2(t,0^+),\mu\},&\text{if }r(t)=r_{\max},
    \end{cases}\\
    &d_B(t)=\begin{cases}
        \mu,&\text{if }r(t)\in (0,r_{\max}],\\
        \min\{D_1(\rho_1(t,0^-)),\mu\},&\text{if }r(t)=0.
    \end{cases}
\end{align*}
Here, $S_i$ and $D_i$ with $i\in\{1,2\}$ are the usual supply and demand functions \cite{lebacque1996godunov} defined for a flux $f_i$ with a single maximum at $\sigma$ by
\begin{align*}
    D_i(\rho)=\begin{cases}
        f_i(\rho),&\rho\leq\sigma,\\
        f_i(\sigma),&\rho>\sigma,
    \end{cases}\qquad
    S_i(\rho)=\begin{cases}
        f_i(\sigma),&\rho\leq\sigma,\\
        f_i(\rho),&\rho>\sigma.
    \end{cases}
\end{align*}
Note that, here $f_i(\rho):=\rho v_i(\rho)$ holds.
Finally, the development of the buffer is described by
\[r'(t)=q_1(t)-q_2(t).\]
For a more detailed discussion of the model and an extension to more general junctions we refer the reader to \cite{herty2009novel}. 

Comparing the local model with \eqref{eq:model}, we see that the non-local version adapts the structure of the local model.
Besides the non-local velocity, one of the main differences between the two models is the definition of the supply and demand functions.
Since the non-local model allows for a higher flow, as the local density at $x$  and the non-local velocity are not directly coupled, the flow is not allowed to exceed $\rho_{\max}^2V_2(t,0^+)$ in contrast to $S(\rho_2(0^+,t))$ at the junction point.
Moreover, extending this fact to the whole interval $[-\eta,0)$ with similar ideas as in \cite{friedrich2020onetoone, friedrich2022network}, we end up with the model \eqref{eq:model}.\\
Even though the local model \eqref{eq:model0herty} strongly inspired the non-local one, it is not a priori clear which local solution will be obtained as $\eta \to 0$.
For example, the 1-to-1 model \eqref{eq:1to1NObuffer} without a buffer seems to converge to the vanishing viscosity solution of a discontinuous conservation law in certain cases, see \cite{friedrich2020onetoone}.\\
For the non-local model \eqref{eq:model}, it becomes apparent that the solutions might be different to the solution of the local model \eqref{eq:model0herty}.
This is induced by the different definition of the supply and demand functions mentioned above.
Hence, when supply or demand in the model \eqref{eq:model0herty} are active differences can occur, since, in general,
\begin{align*}
   &D_1(\rho_1(t,0^-))\neq \rho_1(t,0^-)v_2(\rho_2(t,0^+)=\lim_{\eta\to 0}\rho_1(t,0^-)V_2(t,0^+),\\ & S_2(\rho_2(t,0^+))\leq \rho^2_{\max} v_2(\rho_2(t,0^+))=\lim_{\eta\to 0}\rho_{\max}^2V_2(t,0^+). 
\end{align*}
The fact that the model \eqref{eq:model}, in general, does not converge to solutions of \eqref{eq:model0herty} is also shown by the following example:

\begin{example}\label{lem:counter}
Let us consider constant initial data $\rho_1$ and $\rho_2$ on each road with the same flux function on both roads.
Further, let $S(\rho_2)<D(\rho_1)$ and $S(\rho_2)<\mu$ hold.
In this case, we can exclude that the solution of \eqref{eq:model} converges to the solution of the local model~\eqref{eq:model0herty}.
Therefore, we show that the buffer is different for the model \eqref{eq:model0herty} and \eqref{eq:model}.
We start with the local case.
Here, we have
\[q_1(t)=\min\{\mu,D(\rho_1)\},\quad q_2(t)=\min\{S(\rho_2),\mu\}.\]
Moreover, due to $S(\rho_2)<D(\rho_1)<\mu,$
\[r'(t)=q_1(t)-q_2(t)>0,\] this means that the buffer is increasing.
On the contrary, in the non-local case we have 
\[r'(t)=\min\{\mu,\rho_1(t,0^-) V_2(t,0)\}-\min\{\mu,\rho_1(t,0^-) V_2(t,0),\rho^2_{\max} V_2(t,0)\}.\]
Due to the choice of the same flux function and thanks to the maximum principle we get $$\rho_1(t,x)\leq \rho^1_{\max}=\rho^2_{\max},$$ and hence,
\[r'(t)=\min\{\mu,\rho_1(t, 0^-) V_2(t,0^-)\}-\min\{\mu,\rho_1(t, 0^-) V_2(t,0)\}=0.\]
Therefore, the buffer does not increase for every $\eta>0$, this also holds in the limit $\eta\to 0.$ For this reason, the solution of the non-local problem does not, in general, converge to the solution of the local one. 
\end{example}
Even though convergence to the model \eqref{eq:model0herty} can be excluded for certain initial data, there are still cases where we can observe (numerically) a convergence.

\subsection{Limit model}
Since the local model \eqref{eq:model0herty} seems not to be the correct limit model of \eqref{eq:model}, let us formally derive a potential candidate for the limit model of \eqref{eq:model}.
The limit of $\omega_\eta$ to the Dirac delta needs to be considered in the non-local velocities and the supply function.
Hence, the model \eqref{eq:model} formally converges to 
\begin{equation}
\begin{aligned} \label{eq:model0} 
&\partial_t \rho_1(t,x) + \partial_x \left(\rho_1(t,x) v_1(\rho_1(t,x))\right)=0,\\
&\partial_t \rho_2(t,x) + \partial_x (\rho_2(t,x) v_2(\rho_2(t,x)))=0,\\
& r'(t)=\min\{s_B(t,0^-),\,\rho_1(t, 0^-) v_2(\rho_2(t,0^+)\}-\min\{d_B(t),\,\rho^2_{\max} v_2(\rho_2(t,0^+))\},\\
\end{aligned}
\end{equation}
where
\begin{align*}
&s_B(t,x)=\begin{cases}
    \mu,  \quad &r(t)\in [0, r_{\max}),\\
    \min\{\rho^2_{\max} v_2(\rho_2(t,x)),\, \mu\}, &r(t)= r_{\max},
\end{cases}\\
&d_B(t)=\begin{cases}
        \mu, \quad &r(t)\in (0, r_{\max}],\\
    \min\{\rho_1(t,0^-)  v_2(\rho_2(t,0+)),\,\mu\}, &r(t)= 0.
\end{cases}
\end{align*}
Comparing the dynamics to \eqref{eq:model0herty} we see that the outflow of the first road and the inflow to the second road and therefore the development of the buffer are different.
The reason is similar as before that
\begin{align*}
   D_1(\rho_1(t,0^-))\neq \rho_1(t,0^-)v_2(\rho_2(t,0^+)\quad\text{and}\quad S_2(\rho_2(t,0^+))\leq \rho^2_{\max} v_2(\rho_2(t,0^+))
\end{align*}
holds.\\
Note that we will not prove the well-posedness of \eqref{eq:model0} or the limit from \eqref{eq:model} to \eqref{eq:model0} rigorously.
Nevertheless, we investigate the limit and model \eqref{eq:model0} numerically in the following.
In particular, a numerical convergence to model \eqref{eq:model0} can be obtained.

\subsection{Numerical examples}
During this section the space step size for all examples is chosen as $\Delta x=10^{-3}$ and the time step is chosen according to the CFL condition \eqref{eq:CFL}.
We deal with the following example: $v_1(\rho)=1-\rho$, $v_2(\rho)=1-5\rho/3$, constant initial data of $\rho_{1,0}\equiv 0.75$ and $\rho_{2,0}\equiv 0.5$ with $r_{\max}=\infty$ and $r_0=0$.
If not stated otherwise, we use $\mu=0.15$.
Both roads are in a congested state but due to the lower maximum density of the second road, the latter one is even more congested.
This results in a backwards moving traffic jam on the first road and, depending on the capacity, in an increasing buffer.
\begin{remark}
    We note that the obtained results are very similar in case of a limited buffer, i.e., $r_{\max}<\infty$.
    As soon as the buffer gets full, an additional wave is induced moving backwards on the first road.
    To simplify the discussion and concentrate on the effects induced by the capacity, we stick to an infinite buffer.
\end{remark}
Before considering the limit process, we start with a comparison of the dynamics given by \eqref{eq:model} for different kernels in the situation described above and $\eta=0.5$.
Here, we consider a constant kernel $\omega_\eta(x)=1/\eta$, a linear decreasing kernel, i.e. $\omega_\eta(x)=2(\eta-x)/\eta^2$ and a quadratic decreasing kernel, i.e. $\omega_\eta(x)=3(\eta^2-x^2)/(2\eta^3)$.
The approximate solutions are shown in Figure \ref{fig:1}.
Since most of the interactions take place on the first street, we have increased the visual representation of this road.
Due to the bottleneck situation between the two roads, the buffer and the congestion on the first road are increasing.
In particular, for all models the flow on the first road is higher than the capacity $\mu$ of the buffer.
As can be seen, the linear and quadratic kernels behave very similar and provide a rather continuous increase of the density until drivers are aware of the junction around $x\approx-0.5$ which results in a steeper increase.
However, the constant kernel behaves differently.
There are two effects visible: First, there are additional increases depending on $\eta=0.5$.
This effect is also present without a buffer model, i.e. in the model \eqref{eq:1to1NObuffer} and a constant kernel.
The approximate solution without a buffer can also be seen in Figure \ref{fig:1}.
The obtained behavior is a specific feature of the kernel, since drivers become aware of the bottleneck at $x=-0.5$ and the traffic directly ahead, which results in a traffic jam.
This traffic jam then resolves but there is a point in front of the junction ($x\approx-0.25$) where the current flow rate is influenced by the capacity of the buffer which increases the density.
This is the second effect which can be obtained.
Due to the non-locality this effects propagate backwards.
The effect of the capacity on the flow can also be observed for the other two kernels, even though it is smoothed out slightly.

\setlength{\fwidth}{0.8\textwidth}
\begin{figure}
    \centering
    \input{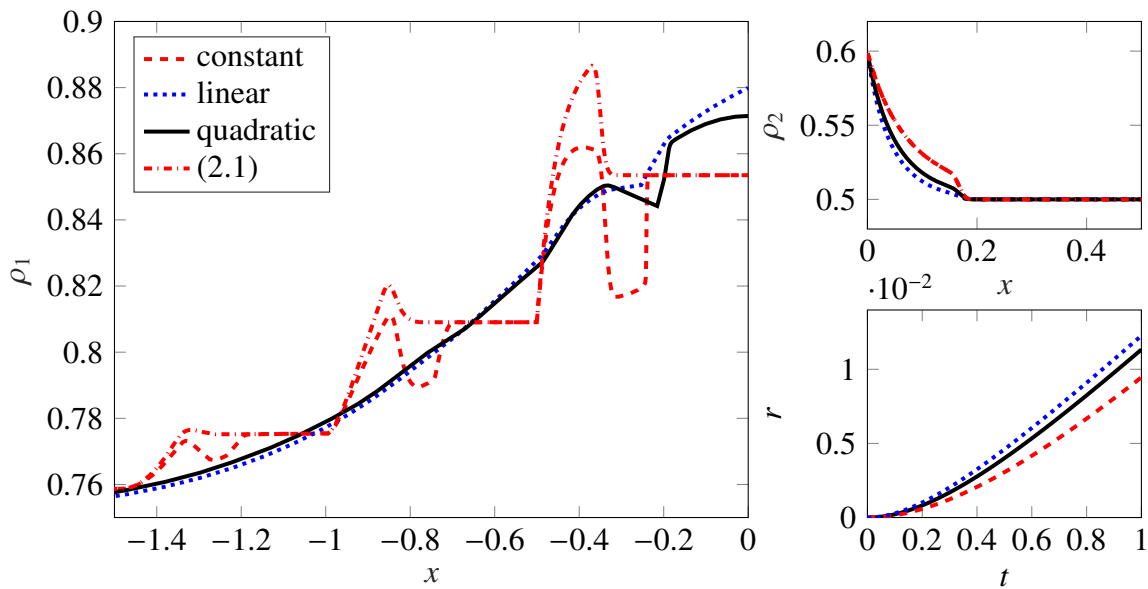}
        \caption{Approximate solution of the non-local buffer model with non-local range of $\eta=0.5$ with different kernels and the approximate solution of \eqref{eq:1to1NObuffer} with a constant kernel at $T=1$: approximate solutions of road 1, left, of road 2, top right and the evolution of the buffer,  bottom right.}
    \label{fig:1}
\end{figure}
Even though the effect of the increasing density for the constant kernel is explainable, we want to consider less effects on the dynamics to better understand the behavior.
Further, as the dynamics for the linear and quadratic kernel are rather similar, we restrict ourselves to the linear kernel for the remaining part of this section.

Next, we turn our attention to the numerical limit $\eta\to 0$.
For the local models \eqref{eq:model0herty} and \eqref{eq:model0} we use a Godunov scheme, where the fluxes at the junction point are defined by the coupling.
In the above mentioned example the capacity is active.
This holds for the local case, too, such that the two local models produce the same solution.
Further, this results in the same outflow of the first road in all models given by $\mu$.
Hence, for a decreasing sequence of $\eta$, the convergence can be numerically obtained, as can be seen in Figure \ref{fig:2}.
Small differences can still be seen on road 2, as the inflow on this road is still different for the models.
We note that this difference may disappear in the limit.
Nevertheless, choosing an even smaller capacity, e.g. $\mu=0.05$, gives the same dynamics for all models.
\begin{figure}
    \centering
    \input{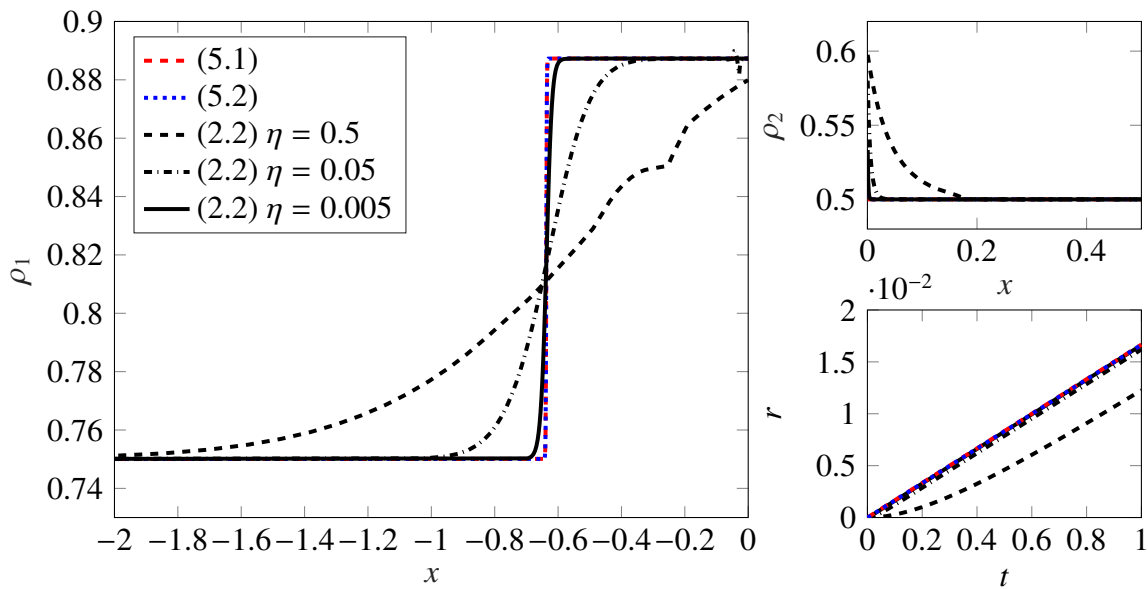}
        \caption{Approximate solutions of the non-local buffer model with different non-local ranges and the approximate solutions of the local models \eqref{eq:model0herty} and \eqref{eq:model0} at $T=1$ and a capacity of $\mu=0.1$ at $T=1$: approximate solutions of road 1, left, of road 2, top right and the evolution of the buffer,  bottom right.}
    \label{fig:2}
\end{figure}
In contrast, increasing $\mu$ demonstrates the possible convergence to \eqref{eq:model0} instead of \eqref{eq:model0herty}.
In Figure \ref{fig:3}, the same example as before, but with $\mu=0.15$, is considered.
In the local model \eqref{eq:model0herty} the outflow of the first road is higher than in the local model \eqref{eq:model0}.
This results in a lower density, but fuller buffer.
The decreasing sequence of $\eta$ shows that \eqref{eq:model0} might be the limit of the non-local buffer model \eqref{eq:model}.

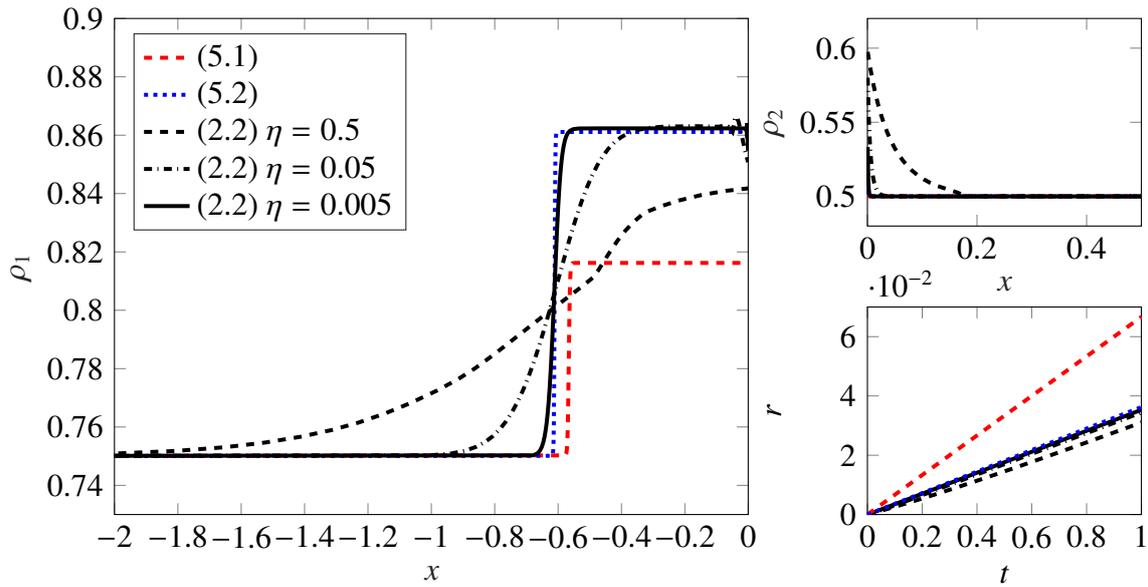
\begin{figure}
    \centering
%
%
\begin{tikzpicture}

\begin{axis}[%
width=0.606\fwidth,
height=0.478\fwidth,
at={(0\fwidth,0\fwidth)},
scale only axis,
xmin=-2,
xmax=0,
xlabel style={font=\color{white!15!black}},
xlabel={$x$},
ymin=0.73,
ymax=0.9,
ylabel style={font=\color{white!15!black}},
ylabel={$\rho_1$},
axis background/.style={fill=white},
legend style={legend cell align=left, align=left, draw=white!15!black, legend pos=north west}
]
\addplot [color=red, dashed, line width=1.5pt]
  table[row sep=crcr]{%
-2.9995	0.75\\
-0.5765	0.750266036368497\\
-0.5745	0.750771372814923\\
-0.5735	0.751307667347457\\
-0.5725	0.752204126354593\\
-0.5715	0.75367998495745\\
-0.5705	0.756050002892725\\
-0.5695	0.759708401667451\\
-0.5685	0.765026106682929\\
-0.5665	0.780576326315478\\
-0.5645	0.797418332166008\\
-0.5625	0.808414379811406\\
-0.5615	0.811466428733813\\
-0.5605	0.813387452334606\\
-0.5595	0.814555731278141\\
-0.5575	0.815660245043571\\
-0.5545	0.816117731854837\\
-0.5465	0.816226406039279\\
-0.0445000000000002	0.816227766016838\\
-0.000500000000000167	0.816227766016838\\
};
\addlegendentry{\eqref{eq:model0herty}}

\addplot [color=blue, dotted, line width=1.5pt]
  table[row sep=crcr]{%
-2.9995	0.75\\
-0.6185	0.750125203468011\\
-0.6165	0.750782851315597\\
-0.6155	0.751942257106099\\
-0.6145	0.754742819192044\\
-0.6135	0.761159696498388\\
-0.6125	0.774259879334257\\
-0.6095	0.841164630112914\\
-0.6085	0.852576306557074\\
-0.6075	0.857743720957382\\
-0.6065	0.859829335258659\\
-0.6055	0.860630160116819\\
-0.6035	0.861044274393523\\
-0.5955	0.861111086612657\\
-0.000500000000000167	0.861111111111111\\
};
\addlegendentry{\eqref{eq:model0}}

\addplot [color=black,dashed, line width=1.5pt]
  table[row sep=crcr]{%
-2.9995	0.750000590994808\\
-2.0105	0.750854686182429\\
-1.8065	0.75179287428664\\
-1.6525	0.753051652189611\\
-1.5255	0.754655006281538\\
-1.4135	0.756635832756599\\
-1.3155	0.75892937846085\\
-1.2305	0.761489935924316\\
-1.1485	0.764539394771195\\
-1.0665	0.768164761301537\\
-0.9815	0.772505606721675\\
-0.9165	0.776330880600435\\
-0.8615	0.780148880064498\\
-0.8005	0.784988517354584\\
-0.6505	0.797737246274394\\
-0.5205	0.808498844085891\\
-0.4895	0.811242142869104\\
-0.4745	0.813461977133895\\
-0.4415	0.81914922500021\\
-0.4125	0.82383944617672\\
-0.3875	0.827294929827685\\
-0.3625	0.830178212164761\\
-0.3365	0.832606254475591\\
-0.3165	0.833896770187068\\
-0.2805	0.835459865510737\\
-0.2165	0.837686026706975\\
-0.1475	0.839503601129131\\
-0.0714999999999999	0.840917038431532\\
-0.000500000000000167	0.841775324456624\\
};
\addlegendentry{\eqref{eq:model} $\eta=0.5$}

\addplot [color=black,dashdotted, line width=1.5pt]
  table[row sep=crcr]{%
-2.9995	0.75\\
-1.0035	0.750286506670533\\
-0.9365	0.750923392864646\\
-0.8925	0.751893452481322\\
-0.8585	0.753205033213324\\
-0.8305	0.754847890828993\\
-0.8065	0.756809879602559\\
-0.7845	0.759180733035682\\
-0.7645	0.7619103049647\\
-0.7455	0.765090678952771\\
-0.7275	0.768692391168461\\
-0.7105	0.77265918915497\\
-0.6935	0.777193610201926\\
-0.6755	0.782606381431659\\
-0.6565	0.788953704111089\\
-0.6355	0.796601893733348\\
-0.6065	0.807854614032465\\
-0.5645	0.824126700095859\\
-0.5435	0.831587913041475\\
-0.5255	0.83737269724289\\
-0.5085	0.842226649142078\\
-0.4925	0.846218236897138\\
-0.4765	0.849647875306508\\
-0.4595	0.852700974098807\\
-0.4415	0.855323181985891\\
-0.4225	0.857490865163055\\
-0.4015	0.859284141403948\\
-0.3775	0.860719235737378\\
-0.3495	0.861785731969986\\
-0.3145	0.86251667749434\\
-0.2645	0.86294151994797\\
-0.1635	0.863092919925965\\
-0.0964999999999998	0.862902260050111\\
-0.0884999999999998	0.862990830132902\\
-0.0714999999999999	0.863475685039559\\
-0.0634999999999999	0.862967339227438\\
-0.0545	0.861767232468084\\
-0.0495000000000001	0.861110656599982\\
-0.0474999999999999	0.862127447505683\\
-0.0425	0.865282138895796\\
-0.0394999999999999	0.866182930176528\\
-0.0365000000000002	0.866240398008143\\
-0.0335000000000001	0.865646535446152\\
-0.0295000000000001	0.864195181983181\\
-0.0225	0.860828446616595\\
-0.00950000000000006	0.854473084540064\\
-0.000500000000000167	0.85075419212349\\
};

\addlegendentry{\eqref{eq:model} $\eta=0.05$}

\addplot [color=black, line width=1.5pt]
  table[row sep=crcr]{%
-2.9995	0.75\\
-0.6785	0.750284145032471\\
-0.6655	0.750919749979607\\
-0.6575	0.751869108654573\\
-0.6515	0.753154341833672\\
-0.6465	0.75484315080062\\
-0.6425	0.756782241949678\\
-0.6385	0.759428742467975\\
-0.6355	0.761998780669925\\
-0.6325	0.765171506995318\\
-0.6295	0.769034579261025\\
-0.6265	0.773659594829433\\
-0.6225	0.781071595172676\\
-0.6185	0.78984168144223\\
-0.6135	0.802251370336444\\
-0.6035	0.827509006978266\\
-0.5995	0.836020859690342\\
-0.5955	0.843044453946414\\
-0.5925	0.847306076875812\\
-0.5895	0.850765099461317\\
-0.5865	0.853515722613038\\
-0.5825	0.856269548277353\\
-0.5785	0.858209609216409\\
-0.5745	0.859553456402336\\
-0.5695	0.860652359896612\\
-0.5625	0.861516147811357\\
-0.5525	0.862053619350027\\
-0.5345	0.862305849197216\\
-0.4615	0.862349211208404\\
-0.00649999999999995	0.862324672565823\\
-0.00550000000000006	0.861510731938782\\
-0.00349999999999984	0.864140408416538\\
-0.00249999999999995	0.862324598894713\\
-0.000500000000000167	0.854661918093421\\
};

\addlegendentry{\eqref{eq:model} $\eta=0.005$}
\end{axis}

\begin{axis}[%
width=0.262\fwidth,
height=0.2\fwidth,
at={(0.72\fwidth,0.278\fwidth)},
scale only axis,
xmin=0,
xmax=0.5,
xlabel style={font=\color{white!15!black}},
xlabel={$x$},
ymin=0.48,
ymax=0.62,
ylabel style={font=\color{white!15!black}},
ylabel={$\rho_2$},
axis background/.style={fill=white}
]
\addplot [color=red, dashed, line width=1.5pt, forget plot]
  table[row sep=crcr]{%
0.000500000000000167	0.5\\
2.9995	0.5\\
};
\addplot [color=blue, dotted, line width=1.5pt, forget plot]
  table[row sep=crcr]{%
0.000500000000000167	0.516666666666667\\
0.00150000000000006	0.5\\
2.9995	0.5\\
};
\addplot [color=black,dashed, line width=1.5pt, forget plot]
  table[row sep=crcr]{%
0.000500000000000167	0.597472185559083\\
0.00450000000000017	0.588168769394733\\
0.00950000000000006	0.578103460537878\\
0.0145	0.569459470436804\\
0.0194999999999999	0.561977164225659\\
0.0245000000000002	0.555456803489137\\
0.0305	0.548684222702793\\
0.0365000000000002	0.542862182716992\\
0.0425	0.537827232274008\\
0.0495000000000001	0.532777689698031\\
0.0565000000000002	0.528464470015991\\
0.0645000000000002	0.524276429394101\\
0.0725000000000002	0.52073922530378\\
0.0815000000000001	0.517395475976733\\
0.0914999999999999	0.514321079399364\\
0.1025	0.511563312689206\\
0.1145	0.509144523914403\\
0.1285	0.506926456064799\\
0.1445	0.504988871913562\\
0.1625	0.502806817432072\\
0.1785	0.500392842219601\\
0.1865	0.500050393494441\\
0.2105	0.500000000405741\\
2.9995	0.5\\
};
\addplot [color=black,dashdotted, line width=1.5pt, forget plot]
  table[row sep=crcr]{%
0.000500000000000167	0.580095908987072\\
0.00249999999999995	0.553153511181687\\
0.00450000000000017	0.536323115681395\\
0.00649999999999995	0.525281407446003\\
0.00850000000000017	0.517809293659632\\
0.0105	0.512648175491919\\
0.0125000000000002	0.509033371225747\\
0.0145	0.506477097519489\\
0.0165000000000002	0.504657134221905\\
0.0194999999999999	0.502849432274904\\
0.0225	0.50174816768196\\
0.0265	0.50091359623202\\
0.0325000000000002	0.500346078172814\\
0.0434999999999999	0.500058554588237\\
0.0825	0.500000107981207\\
2.9995	0.5\\
};
\addplot [color=black, line width=1.5pt, forget plot]
  table[row sep=crcr]{%
0.000500000000000167	0.533340230661428\\
0.00150000000000006	0.512365270736548\\
0.00249999999999995	0.504753440346409\\
0.00349999999999984	0.501851796297876\\
0.00450000000000017	0.500725104932283\\
0.00649999999999995	0.500111708650214\\
0.0114999999999998	0.500001044970527\\
0.4655	0.5\\
2.9995	0.5\\
};
\end{axis}

\begin{axis}[%
width=0.262\fwidth,
height=0.2\fwidth,
at={(0.72\fwidth,0\fwidth)},
scale only axis,
xmin=0,
xmax=1,
ymin=0,
ymax=0.07,
xlabel={$t$},
ylabel={$r$},
axis background/.style={fill=white}
]
\addplot [color=black,dashed, line width=1.5pt, forget plot]
  table[row sep=crcr]{%
0	0\\
0.04	0.00101752630535068\\
0.0800000000000001	0.00207064500973675\\
0.12	0.00315539262490172\\
0.17	0.00455073823819596\\
0.22	0.00598475108053464\\
0.27	0.00745276270095374\\
0.33	0.00925367139681099\\
0.39	0.0110919162167016\\
0.46	0.0132770656368968\\
0.53	0.0154994627839771\\
0.61	0.0180776912568095\\
0.7	0.0210187008958413\\
0.8	0.0243272091795981\\
0.91	0.028005745637588\\
1	0.0310396067845531\\
};

\addplot [color=black,dashdotted, line width=1.5pt, forget plot]
  table[row sep=crcr]{%
0	0\\
0.01	0.000257696616772884\\
0.03	0.00082788799251321\\
0.05	0.00144608962707338\\
0.0800000000000001	0.00242492194455024\\
0.12	0.0037780645057095\\
0.19	0.00619369969813399\\
0.38	0.0128041873173521\\
1	0.0343967341076181\\
};

\addplot [color=black, line width=1.5pt, forget plot]
  table[row sep=crcr]{%
0	0\\
0.01	0.000305916884149715\\
0.05	0.00171473436527969\\
1	0.035315954200354\\
};
\addplot [color=red, dashed, line width=1.5pt, forget plot]
  table[row sep=crcr]{%
0	0\\
1	0.0666666666666667\\
};
\addplot [color=blue, dotted, line width=1.5pt, forget plot]
  table[row sep=crcr]{%
0	0\\
0.01	0.000343153553119091\\
1	0.0362459313306784\\
};
\end{axis}

\end{tikzpicture}%
    \caption{Approximate solutions of the non-local buffer model with different non-local ranges and the approximate solutions of the local models \eqref{eq:model0herty} and \eqref{eq:model0} at $T=1$ and a capacity of $\mu=0.15$ at $T=1$: approximate solutions of road 1, left, of road 2, top right and the evolution of the buffer,  bottom right.}
    \label{fig:3}
\end{figure}

\begin{remark}
We remark that even without a buffer the non-local model does not necessarily converge to the solution which maximizes the flux through the junction.
In particular, for a local 1-to-1 junction without a buffer there a several possible Riemann solvers available at the interface, see e.g. \cite{adimurthi2005optimal} and the references therein. 
In \cite{friedrich2020onetoone}, it has been already discovered that the non-local 1-to-1 model without a buffer seems to converge to the vanishing viscosity solution and not to the solution given by supply and demand which maximizes the flux through the junction.
Hence, we cannot expect a convergence to the model \eqref{eq:model0herty} which also maximizes the flux through the junction for every case, in particular, if the buffer is not active.
\end{remark}

\section{An infinite non-local reach}
Besides the limit $\eta\to 0$ also the case $\eta\to\infty$ is of great interest.
This limit represents the situation in which drivers have perfect knowledge about the downstream traffic.
In~\cite{chiarello2018global}, it has been proven that on a single road a non-local traffic flow model tends to the solution of a linear transport equation with transport velocity $V_{\max}$.
The limit for non-local traffic networks without buffers has been already investigated in \cite{friedrich2022network}.
Here, the limit equations share similarities to a model which is usually applied in the context of supply chains~\cite{ArmbrusterDegondRinghofer2006}.
The key characteristic is that cars travel with a constant velocity until a certain capacity is reached.
This is very similar to what happens when goods pass through a processor, cf. the  supply chain models introduced in \cite{ArmbrusterDegondRinghofer2006, goettlich2010supplychains}.
Therefore, we start by discussing the solution concept of this model.

\subsection{A supply chain model}
Following \cite{ArmbrusterDegondRinghofer2006, goettlich2010supplychains} two processors with one buffer in between can be described by:
\begin{equation}\label{eq:prodmodel}
    \begin{aligned}
       & \partial_t \rho_1 +\partial_x \min\{\rho_1v,\mu_1\}=0\\
    &\partial_t \rho_2 +\partial_x \min\{\rho_1v,\mu_2\}=0\\
    &r'(t)=\min\{\rho_1(t,0^-)v,\mu_1\}-d_B(t)\\
    &d_B(t)=\begin{cases} \min\{\rho_1(t,0^-)v,\mu_1,\mu_2\},\quad &\text{if }r(t)=0,\\
    \mu_2,\quad &\text{if }r(t)>0.
    \end{cases}
\end{aligned}
\end{equation}
Note that we consider an unlimited buffer for simplicity.
Further, the buffer have different capacities, but the same processing velocity.
In the case $\mu_1=\mu_2$ and $r_0=0$ we can view the problem as a single processor as the buffer will never fill up.
For the initial data
\begin{equation}
    \rho_0(x)=\begin{cases} \rho_1,\quad &x<0,\\
    \rho_2, & x>0,
    \end{cases}
\end{equation}
The solution is
\begin{itemize}
    \item a linear transport with velocity $v$, if $\rho_i<\mu_i/v$ $i\in\{1,2\}$, or
    \item a rarefaction wave with intermediate state $\mu_2/v$, if $\rho_1>\mu_1/v$ and $\rho_2<\mu_2/v$.
\end{itemize}
In the latter case the capacity of the buffer is not sufficient to handle the complete flux, such that the flux decreases to the capacity.
In \cite{friedrich2022network}, it is shown that this situation, i.e. \eqref{eq:prodmodel} with $\mu_1=\mu_2$ and $r_0=0$, is the limit of the 1-to-1 model \eqref{eq:1to1NObuffer} without a buffer.

Now, we consider the more general case, $\mu_1\neq \mu_2$.
For simplicity, the buffer is initially empty and we consider only the case $\rho_i<\mu_i/v$ $i\in\{1,2\}$, since this case is the most important one for our upcoming limit cases.
Explicit formulas for the solution in all cases can be found in \cite[Eq. (4.21)]{goettlich2010supplychains}.
The solution now depends on the outgoing flux of the first processor. 
If it is small enough, all goods can be transported, otherwise only up to the capacity $\mu_2$.
In fact, we have
\begin{itemize}
    \item for $\rho_1\leq\mu_2/v$ a linear transport and 
    \item for $\rho_1>\mu_2/v$ a rarefaction wave with intermediate state $\mu_2/v$
\end{itemize}
as a solution.
In the latter case, the buffer starts to fill.
 
\subsection{Limit models}\label{subsec:limitinftymodel}
Now, we derive the limit models of \eqref{eq:model}.
Therefore, we assume that the initial data $\rho_{1,0}$ and $\rho_{2,0}$ are of compact support.
Note that due to the finite speed of propagation also the solutions $\rho_1(t,\cdot)$ and $\rho_2(t,\cdot)$ are compactly supported.
We can use the following result proven in \cite[Lemma 4.1]{friedrich2022network}:
in the $L^1_{loc}$ limit the non-local velocities converge to
\begin{align}\label{eq:velconv}
    V_1(t,x)\to 0\qquad\text{and}\qquad V_2(t,x)\to v_2(0)\qquad\text{for }\eta \to \infty.
\end{align}
Using similar arguments we can deduce
\begin{align}
    & s_B(t,x)\to \begin{cases}
        \mu\quad & r(t)\in[0,r_{\max}),\\
        \min\{\rho_{\max}^2 v_2(0), \mu\}\quad & r(t)=r_{\max},
    \end{cases}
    \qquad\text{and}\\
    & d_B(t)\to \begin{cases}
        \mu\quad & r(t)\in (0,r_{\max}],\\
        \min\{\rho_1(t,0^-) v_2(0), \mu\}\quad & r(t)=0,
    \end{cases}\qquad\text{for }\eta \to \infty. \label{eq:inflowlimit}
\end{align}
To discuss the limit models in detail, we distinguish two cases depending on the maximum flux on each road.
In the limit the maximum flux of the first road is given by the capacity $\mu$ and of the second road by $\rho_{\max}^2v_2(0)$.
\paragraph{\textbf{Case 1: $\mu\leq\rho_{\max}^2 v_2(0)$}} In this case the limit equations are given by
    \begin{align*}
        &\partial_t \rho_1(t,x) +\partial_x \min\{\rho_1(t,x)v_2(0),\mu\}=0\\
        &\partial_t \rho_2(t,x) +\partial_x \rho_2(t,x)v_2(0)=0,
    \end{align*}
where the inflow on road 2 is given by \eqref{eq:inflowlimit} and hence depends on the buffer size.
If we assign an artificial capacity to the second road which is greater or equal to the maximum flow, i.e. $\mu_2\geq\rho_{\max}^2v_2(0)$, the dynamics are described by
  \begin{align*}
      \partial_t \rho_2(t,x) +\partial_x \min\{\rho_2(t,x)v_2(0),\mu_2\}=0.
  \end{align*}
  Obviously, this does not change the dynamics on the second road, but now the similarities with the model \eqref{eq:prodmodel} become apparent.
  However, the inflow on road 2 and the capacity of the buffer given in \eqref{eq:inflowlimit} do not match the 'artificial' capacity of the second road and are smaller, i.e. $\mu\leq \mu_2$.
  Thus, the model can be seen as an extension of the production model \eqref{eq:prodmodel}, which allows a smaller capacity of the buffer.
  Such models have been studied in \cite{kathinkadiss}.
   Note that once $r(t^*)=0$ for some $t^*\geq 0$ holds, the buffer stays empty for $t\geq t^*$.
   In this case, the complete limit dynamics coincide with \eqref{eq:prodmodel}.
   In contrast, if $r(t)>0$, the buffer decreases.
   \begin{example}
    To get a better understanding of the dynamics of the limit model, we focus on a specific initial condition with an empty initial buffer and infinite capacity of the buffer, i.e.
\begin{align}\label{eq:example}
    \rho_{1,0}(x)=\begin{cases}
        \rho_1,\qquad &\text{if }x\in[a,b],\\
        0,&\text{ else}
    \end{cases}\qquad \rho_{2,0}(x)=0,\qquad r(0)=0,\ r_{\max}=\infty. 
\end{align}
for $a<b<0$. 
The buffer model has a capacity of $\mu$ on the first road.
This capacity is less than the capacity on the second road.
Hence, the second road can take all the flux coming from the first road and the buffer does not increase. 
Further, due to the empty initial buffer, the buffer remains empty.
The solution depends on the value of $\rho_1$:
    \begin{enumerate}
    \item $\mu/v_2(0)>\rho_1$: Here, we have a linear transport.
     \item $\rho_1>\mu/v_2(0)$: The solution of the buffer model is a rarefaction wave with constant state $\mu/v_2(0)$ at $b$ and a shock at $a$, which moves onto the second road. Both the shock at $a$ and the rarefaction wave move at the speed $\mu/\rho_1<v_2(0)$ on the first road.
     On the second road the speed increases to $v_2(0)$.
     \end{enumerate}
   \end{example}
\paragraph{\textbf{Case 2} $\mu>\rho_{\max}^2 v_2(0)$} In this case not only the inflow onto road 2 depends on the buffer size but the flux on the first road, too.
The model equations for the first road are given by
\begin{align*}
    \begin{cases}
    \partial_t \rho_1 +\partial_x \min\{\rho_1v_2(0),\mu\}=0,\quad &\text{ if }r(t)\in[0,r_{\max})\\
    \partial_t \rho_1 +\partial_x \min\{\rho_1v_2(0),\rho_{\max}^2v_2(0)\}=0,\quad &\text{ if }r(t)=r_{\max}.
    \end{cases}
\end{align*}
and the second road is the same to the previous case:
\begin{align*}
    \partial_t \rho_2 +\partial_x \rho_2 v_2(0)=0.
\end{align*}
More importantly, the inflow onto road 2 is given by 
\begin{align*}
\begin{cases}
    \rho_{\max}^2v_2(0),\quad &\text{ if }r(t)\in (0,r_{\max}],\\
    \min\{\rho_1(t,0-)v_2(0),\rho_{\max}^2v_2(0)\},\quad &\text{ if }r(t)=0.
\end{cases} 
\end{align*}
This results in the following development of the buffer in the limit:
\begin{align*}
    r'(t)=\begin{cases}
       \min\{\mu,\rho_1(t,0-)v_2(0)\}-\min\{\rho_1(t,0-)v_2(0),\rho_{\max}^2v_2(0)\}, \quad &\text{if }r(t)=0,\\
       \min\{\mu,\rho_1(t,0-)v_2(0)\}-\rho_{\max}^2v_2(0), \quad &\text{if }r(t)\in (0,r_{\max}),\\
       \min\{\rho_{\max}^2v_2(0),\rho_1(t,0-)v_2(0)\}-\rho_{\max}^2v_2(0), \quad &\text{if }r(t)=r_{\max}.
    \end{cases}
\end{align*}
As can be seen, the dynamics are more involved and strongly dependent on the buffer size.
Nevertheless, for $r(t)\in [0,r_{\max})$ the limit equations can be viewed as the supply chain model \eqref{eq:prodmodel} by assigning the artificial capacity $\mu_2=\rho_{\max}^2v_2(0)$ to the second road.
In particular, the capacity of the buffer is given by $\rho_{\max}^2v_2(0)$ and is equal to the capacity of the next processor as in \eqref{eq:prodmodel}.
Note that the dynamics change when $r(t)=r_{\max}$. 
Models with finite buffers are not studied in \cite{ArmbrusterDegondRinghofer2006,goettlich2010supplychains}.
\begin{example}
    Let us return to the setting \eqref{eq:example} for the second case.
    Now, the capacity of the first road is higher than that of the second one.
This allows for an increasing buffer and interesting dynamics:
\begin{enumerate}
        \item $\mu/v_2(0)>\rho_{\max}^2>\rho_1$: Here, the density is low enough not to affect the buffer, so the solution is simply the linear transport.
        \item $\mu/v_2(0)>\rho_1>\rho_{\max}^2$: The solution is first given by a linear transport equation. 
        When the density reaches the buffer at $t^*=-b/v_2(0)$, the buffer starts to fill at the rate:
        \[r'(t)=v_2(0)(\rho_1-\rho_{\max}^2).\]
        The inflow on the second road is then $\rho_{\max}^2v_2(0)$.
        This results in a density of $\rho_{\max}^2$.
     \item $\rho_1>\mu/v_2(0)>\rho_{\max}^2$: Instead of a linear transport, a rarefaction wave with intermediate state $\mu/v_2(0)$ and a shock is generated. Again, after some time (i.e. when the rarefaction wave reaches $x=0$), the buffer begins to fill with 
\[r'(t)=\mu-v_2(0)\rho_{\max}^2.\]
On road 2 the density is $\rho_{\max}^2$ as before.    
    \end{enumerate}
\end{example}

\subsection{Numerical examples}
In this subsection we consider some numerical results for the different cases discussed before.
For the numerical tests we use $\Delta x=0.01$.

We start with an explicit example for the initial condition \eqref{eq:example} and the parameters chosen such that $\rho_1>\mu/v_2(0)>\rho_{\max}^2$ holds.
More precisely, we choose the initial conditions
\begin{align*}
    \rho_{1,0}(x)=\begin{cases}
        1,\qquad &\text{if }x\in[-5,-1/3],\\
        0,&\text{ else}
    \end{cases}\qquad \rho_{2,0}(x)=0,\qquad r(0)=0,\ r_{\max}=\infty,
\end{align*}
$\mu=0.75$, $v_1(\rho)=1-\rho$, $v_2(\rho)=1-2\rho$ and hence $\rho_{\max}^2=0.5$.
The explicit solution for the limit model for $\eta \to \infty$ is for $t<1/3$ given by:
\begin{align*}
    \rho_{1}(t,x)=\begin{cases}
        0.75,\qquad &\text{if }x\in[-1/3,-1/3+t],\\
        1,\qquad &\text{if }x\in[-5+0.75t,-1/3],\\
        0,&\text{ else}
    \end{cases}\qquad \rho_{2,0}(x)=0,\qquad r(t)=0,
\end{align*}
and for $t\in[1/3,56/9)$ by
\begin{align*}
    \rho_{1}(t,x)&=\begin{cases}
        0.75,\qquad &\text{if }x\in[-1/3,0],\\
        1,\qquad &\text{if }x\in[-5+0.75t,-1/3],\\
        0,&\text{ else}
    \end{cases}\qquad \rho_{2,0}(x)=\begin{cases}
    0.5,\qquad &\text{if }x\in[0,t-1/3],\\
         0,&\text{ else}
    \end{cases},\\
    r(t)&=0.25(t-1/3).
\end{align*}
Note that the characteristics of the limit model, which emerge in the interval $[-5,-1/3]$, have an infinite slope, otherwise the slope is $v_2(0)$.\\
We compare the exact limit solution with a numerical approximation of the non-local buffer model computed with $\eta=2$ for a constant and linear decreasing kernel.
Figure \ref{fig:Infcomparison} shows also the possible limit model for $\eta \to 0$, i.e., \eqref{eq:model0}.
It can be seen that for the different models induced by $\eta$ more mass seems to be transported on the first road, into the buffer and trough the junction with increasing $\eta$.
\setlength{\fwidth}{0.8\textwidth}
\begin{figure}
    \centering
	\input{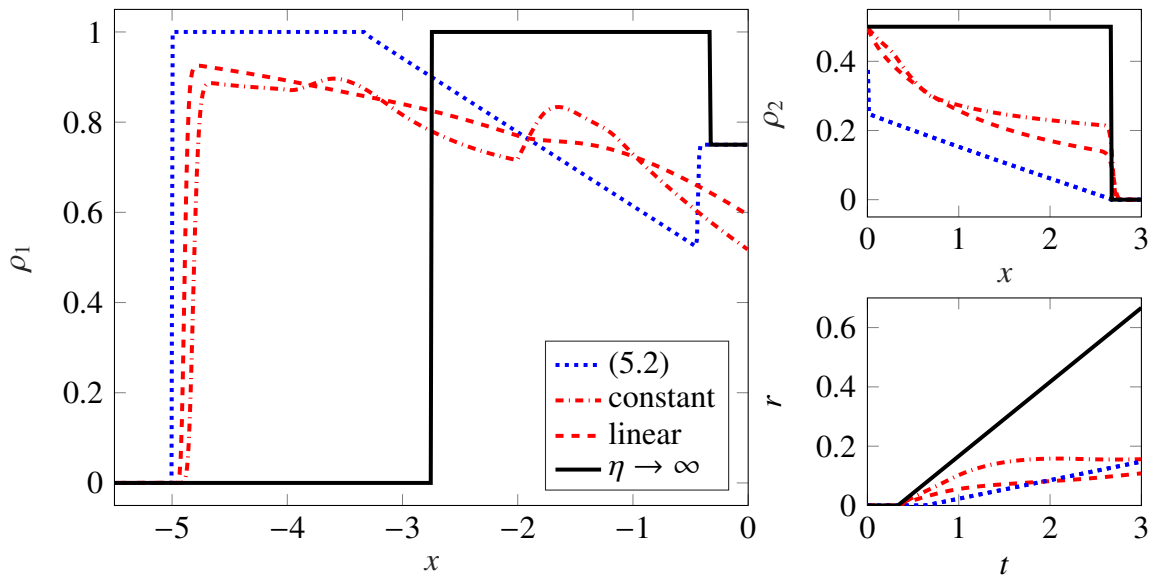}
    \caption{Approximate solution of the non-local buffer model with a constant and linear kernel for $\eta=2$ with the limit models to zero and to infinity at $T=3$: approximate solutions of road 1, left, of road 2, top right and the evolution of the buffer,  bottom right.}
    \label{fig:Infcomparison}
\end{figure}

The characteristics \eqref{eq:char_1}--\eqref{eq:char_2} for $\eta=2$ are shown in Figure \ref{fig:chareta2}. They are approximated by an explicit Euler scheme and the numerical solutions of $\rho_1$ and $\rho_2$. As it can be seen the characteristics are not crossing such that existence and uniqueness of weak solutions might be provable as discussed in Remark \ref{rem:completemodel}. Moreover, some properties of the solutions can be observed such as the increase in the density at $x\approx-5$.
\setlength{\fwidth}{0.8\textwidth}
\begin{figure}
    \centering
	\input{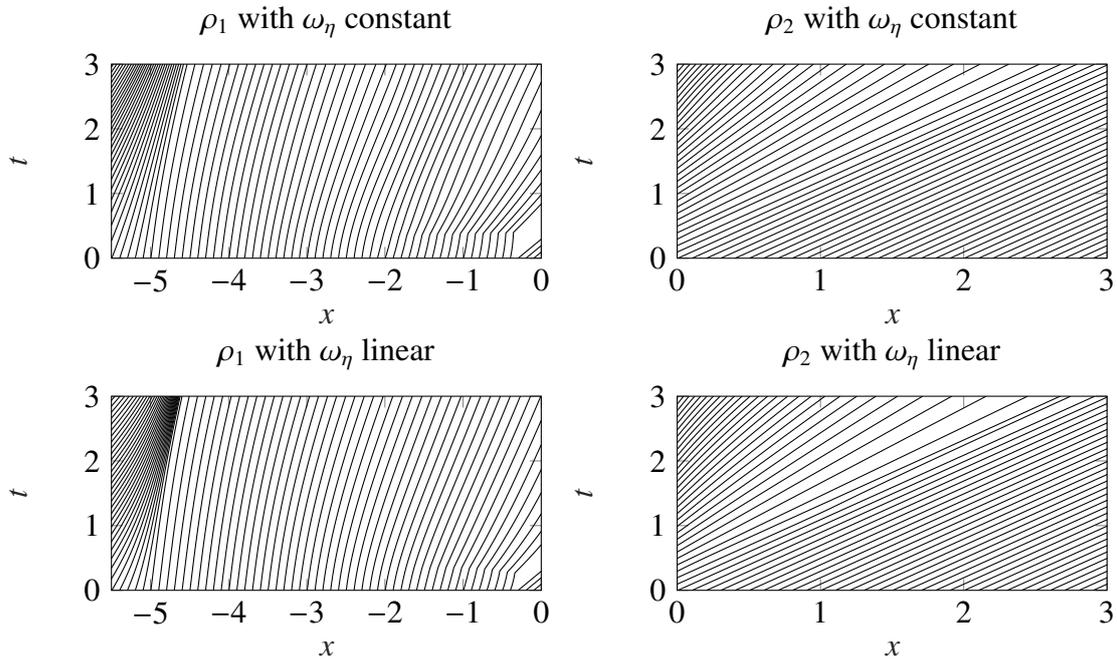}
    \caption{Approximate solutions of the characteristics for the non-local model \eqref{eq:model} with $\eta=2$: approximate solutions on road 1, left column, on road 2, right column with the constant kernel, top row, and a linear decreasing kernel,  bottom row.}
    \label{fig:chareta2}
\end{figure}

In the following, we again restrict ourselves to a linear decreasing kernel.
The limit for $\eta \to \infty$ is shown in Figure \ref{fig:Bufferlim}.
It can be seen that in particular the approximation of the shock and the rarefaction wave of the exact solution for $\eta\to\infty$ requires rather large values of $\eta$. 
Further, with increasing $\eta$ more mass is transported on the first and second road.
\setlength{\fwidth}{0.8\textwidth}
\begin{figure}
    \centering
	\input{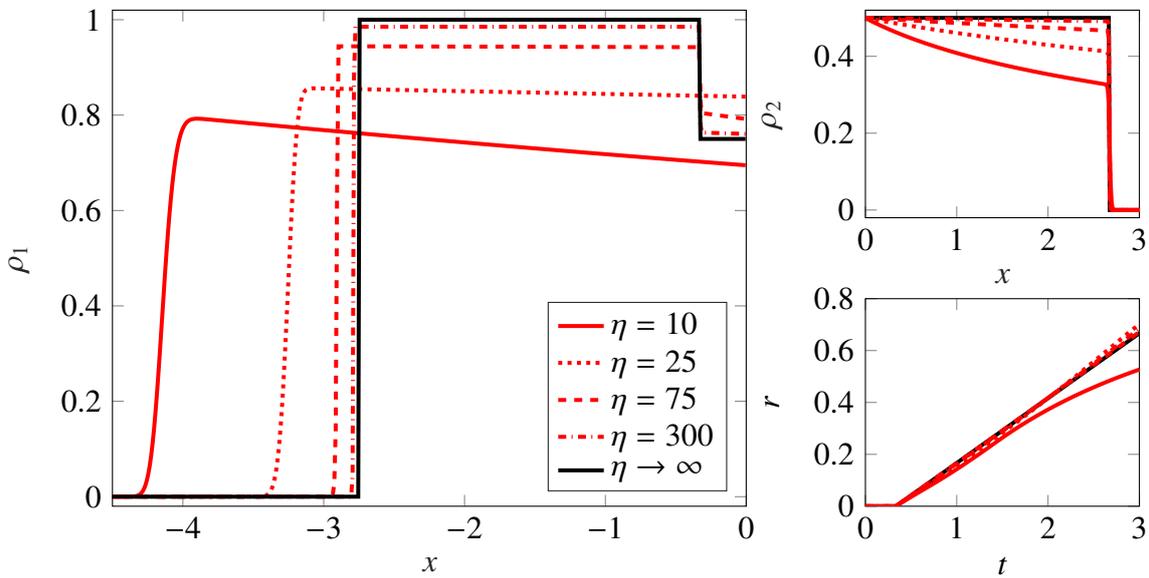}
    \caption{Approximate solution of the non-local buffer model with increasing values of $\eta$ at $T=3$: approximate solutions of road 1, left, of road 2, top right and the evolution of the buffer,  bottom right.}
    \label{fig:Bufferlim}
\end{figure}

Interestingly, convergence can also be obtained for the characteristics \eqref{eq:char_1}--\eqref{eq:char_2}. Using \eqref{eq:velconv}--\eqref{eq:inflowlimit} they converge at least formally to
\begin{align*}
    \xi^1[t,x](\tau)=&x+\int_t^\tau \begin{cases}
        v_2(0)&\text{if } \rho_1(t,\xi^1[t,x](s))v_2(0)\leq \mu\\
        0&\text{else}
    \end{cases}ds,\quad
    \xi^2[t,x](\tau)=&x+(\tau-t)v_2(0).
\end{align*}
This can be seen in Figure \ref{fig:chareta300}, too.
In contrast, to the local case the characteristics do not seem to cross, but they change their slope as soon as the shock reaches them.
Note that this observation is different to the limit $\eta\to 0$ in which a convergence or even a similarity of the characteristics can not be obtained (as in the considered non-local model characteristics always have a positive slope in contrast to the local models \eqref{eq:model0herty} abd \eqref{eq:model0}).
\begin{figure}
    \centering
    \input{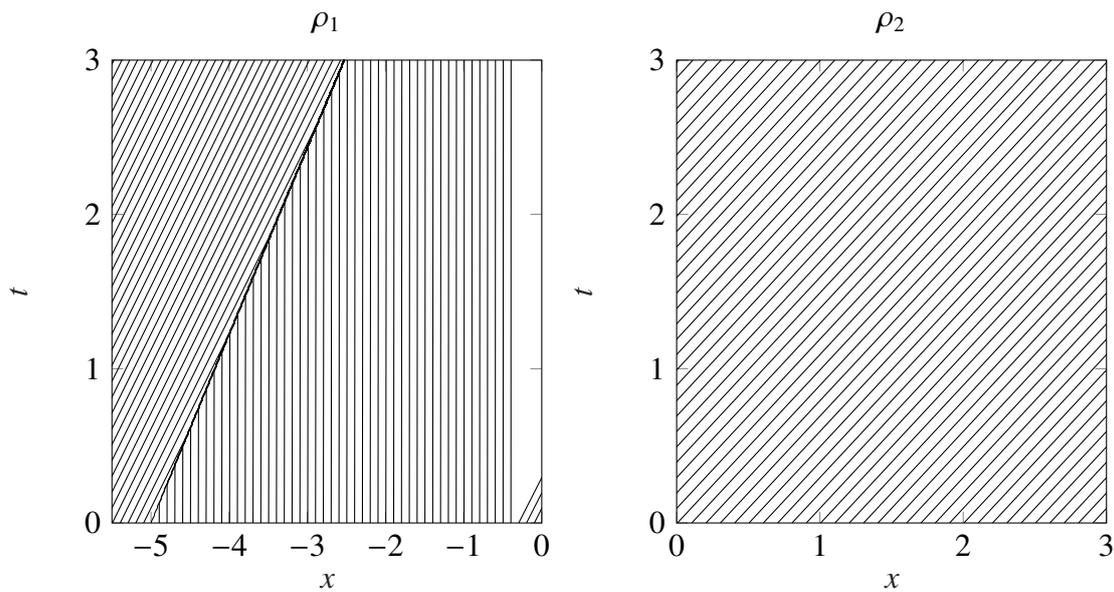}
    \caption{Approximate solutions of the characteristics for the non-local model \eqref{eq:model} with $\eta=300$.}
    \label{fig:chareta300}
\end{figure}

During the discussion of the limit models in Section \ref{subsec:limitinftymodel}, we also encountered specific cases of the production model \eqref{eq:prodmodel} that have not yet been studied in the literature.
Therefore, we consider slight modifications of the initial setup in \eqref{eq:example}.
In Figure \ref{fig:3cunlimited}, we consider the same situation as above but with an unlimited and a limited buffer of $r_{\max}=0.15$ for $\eta=200$ at $T=2$.
In the latter case, the buffer increases until about $T=1$.
Before $T\approx 1$, the solution is the same as discussed above.
Afterwards the dynamics with a limited buffer change, which can be seen in the different speeds for the shock on the first road.
For a full buffer, the speed drops from $0.75$ to $0.5$.
On the second road the solutions for both models coincide.
\setlength{\fwidth}{0.8\textwidth}
 \begin{figure}
    \centering
%
%
\definecolor{mycolor1}{rgb}{0.00000,0.44700,0.74100}%
\definecolor{mycolor2}{rgb}{0.85000,0.32500,0.09800}%
\begin{tikzpicture}

\begin{axis}[%
width=0.606\fwidth,
height=0.478\fwidth,
at={(0\fwidth,0\fwidth)},
scale only axis,
xmin=197,
xmax=201,
xtick={197,198,199,200,201},
xticklabels={-4,-3,-2,-1,0},
xlabel style={font=\color{white!15!black}},
xlabel={$x$},
ymin=-0.02,
ymax=1.02,
ylabel style={font=\color{white!15!black}},
ylabel={$\rho_1$},
axis background/.style={fill=white},
legend style={at={(0.26,0.05)}, anchor=south west, legend cell align=left, align=left, draw=white!15!black}
]
\addplot [color=black, line width=1.5pt]
  table[row sep=crcr]{%
196.995	1\\
200.665	1\\
200.675	0\\
200.995	0\\
};
\addlegendentry{Initial data}

\addplot [color=red, dashed, line width=1.5pt]
  table[row sep=crcr]{%
196.995	0\\
197.165	3.94134437442517e-06\\
197.175	0.00140865741704488\\
197.195	0.98798245372069\\
197.235	0.988030719098447\\
200.665	0.987823412540791\\
200.675	0.787439958868845\\
200.685	0.753394699938212\\
200.695	0.751135610934057\\
200.725	0.750659188844111\\
200.995	0.747095737623084\\
};
\addlegendentry{\eqref{eq:model} $r_{\max}=0.15$}

\addplot [color=blue, dotted, line width=1.5pt]
  table[row sep=crcr]{%
196.995	0\\
197.425	0.000105890264080699\\
197.435	0.00328150928399396\\
197.445	0.0740584217030573\\
197.455	0.93942931567787\\
197.465	0.985397064688925\\
200.665	0.985161705963321\\
200.675	0.784398417314748\\
200.685	0.762480071342139\\
200.995	0.758299248865143\\
};
\addlegendentry{\eqref{eq:model} $r_{\max}=\infty$}

\end{axis}

\begin{axis}[%
width=0.262\fwidth,
height=0.2\fwidth,
at={(0.72\fwidth,0.278\fwidth)},
scale only axis,
xmin=0,
xmax=2,
xlabel style={font=\color{white!15!black}},
xlabel={$x$},
ymin=0,
ymax=0.6,
ylabel style={font=\color{white!15!black}},
ylabel={$\rho_2$},
axis background/.style={fill=white}
]
\addplot [color=black, line width=1.5pt,forget plot]
  table[row sep=crcr]{%
0.00499999999999989	0\\
2.005	0\\
};
\addplot [color=blue, dotted, line width=1.5pt,forget plot]
  table[row sep=crcr]{%
0.00499999999999989	0.499949593951516\\
1.495	0.492605316818001\\
1.655	0.491770615867621\\
1.665	0.484049854221689\\
1.675	0.0203010764357581\\
1.685	0.000422880448655238\\
1.695	5.79353774909919e-06\\
1.985	0\\
2.005	0\\
};
\addplot [color=red, dashed, line width=1.5pt,forget plot]
  table[row sep=crcr]{%
0.00499999999999989	0.499949593951516\\
1.495	0.492605316818001\\
1.655	0.491770405324988\\
1.665	0.484037264897154\\
1.675	0.0203133310339845\\
1.685	0.000423415328605792\\
1.695	5.80479269718381e-06\\
1.985	0\\
2.005	0\\
};
\end{axis}

\begin{axis}[%
width=0.262\fwidth,
height=0.2\fwidth,
at={(0.72\fwidth,0\fwidth)},
scale only axis,
xmin=0,
xmax=2,
ymin=0,
ymax=0.45,
ylabel = {$r$},
xlabel style={font=\color{white!15!black}},
xlabel={$t$},
axis background/.style={fill=white},
legend style={legend cell align=left, align=left, draw=white!15!black}
]
\addplot [color=blue, dotted, line width=1.5pt,forget plot]
  table[row sep=crcr]{%
0	0\\
0.33	4.89226199551496e-07\\
0.36	0.00741631196994064\\
0.61	0.0693785294943186\\
0.85	0.128855409055357\\
1.03	0.173644073665565\\
1.21	0.218592412578916\\
1.39	0.263699979871921\\
1.57	0.308966331777932\\
1.75	0.354391026217665\\
1.93	0.399973622783676\\
2	0.417742737028953\\
};

\addplot [color=red, dashed, line width=1.5pt,forget plot]
  table[row sep=crcr]{%
0	0\\
0.33	4.89226199551496e-07\\
0.34	0.00245292673317099\\
0.49	0.0396558216556113\\
0.64	0.0768034997246443\\
0.76	0.106521095505494\\
0.86	0.131339471561346\\
0.94	0.151229259802698\\
2	0.151229259802698\\
};

\end{axis}

\end{tikzpicture}%
    \caption{Approximate solution of the non-local buffer model at $T=2$ with a finite and infinite buffer: approximate solutions of road 1, left, of road 2, top right and the evolution of the buffer,  bottom right.}
    \label{fig:3cunlimited}
\end{figure}
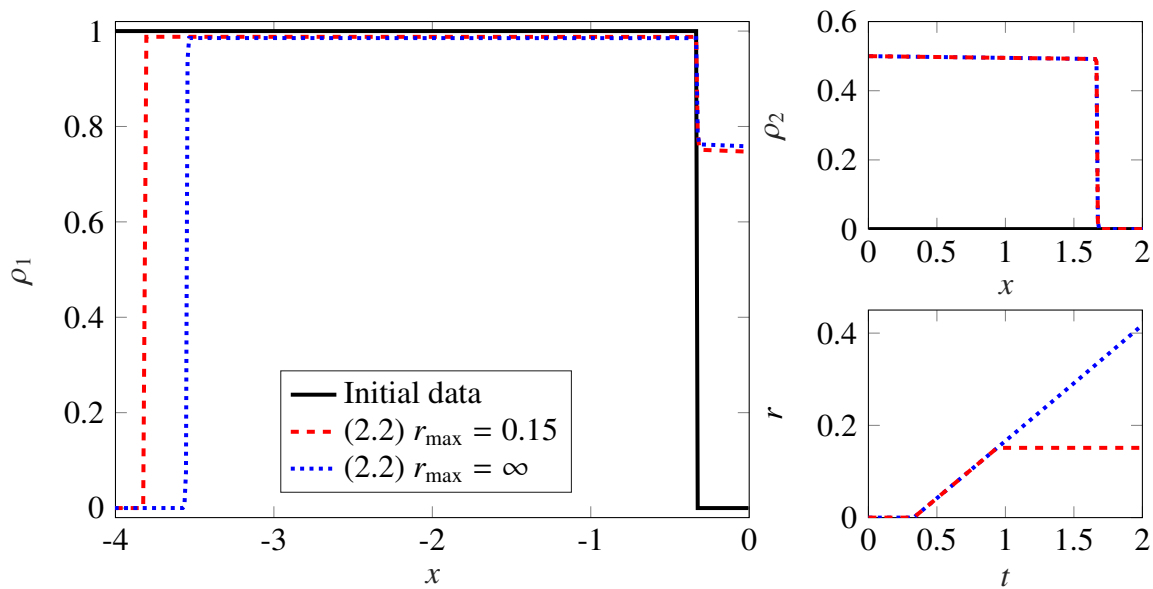


\section{Conclusion}
We presented a non-local buffer model for a 1-to-1 junction with a suitable numerical discretisation.
We have proved that the numerical scheme preserves the maximum principle uniform in $\eta$.
Further, we investigate the model hierarchies induced by the non-local term.
For $\eta\to 0$, we can exclude the convergence to the local model of \cite{herty2009novel} which originally inspired the non-local buffer model.
We propose a different local model which seems to be (numerically) the correct limit of the non-local buffer model.
Moreover, we investigate the limit of $\eta\to \infty$.
Here, drivers drive with constant speed up to a certain capacity.
This shares similarities to supply chain models with two processors and one buffer.
Nevertheless, the obtained limit dynamics are more general than the supply chain models studied in the literature so far.
Again the convergence can be obtained numerically.

Future work could include proofing the existence and uniqueness of weak solutions as hinted in Remark \ref{rem:completemodel} for the complete model \eqref{eq:model}, rigorously proofing the limits of the non-local buffer model and extending the junction model to more complex junctions and network structures.

\section*{Acknowledgement}
F.A.C. is member of Gruppo Nazionale per l’Analisi Matematica, la Probabilit`a e le loro Applicazioni (GNAMPA) of the Istituto Nazionale di Alta Matematica (INdAM). F.A.C. was partially supported by the INdAM-GNAMPA project CUP\_E53C22001930001.
Jan Friedrich is supported by the German Research Foundation (DFG) through SPP 2410 'Hyperbolic Balance Laws in Fluid Mechanics: Complexity, Scales, Randomness' under grant FR 4850/1-1, and Simone G\"ottlich under grant GO 1920/12-1.

\end{document}